\documentclass{article}
\usepackage{amsmath,amssymb,amsthm,mathrsfs}
\usepackage{graphicx}

\title{The Higher Rank Rigidity Theorem \\ for Manifolds With No Focal Points}
\date{\today}
\author{Jordan Watkins \thanks{Partially supported by NSF RTG grant DMS-0602191}\\ University of Michigan}

\DeclareMathOperator{\rank}{rank}

\DeclareMathOperator{\Isom}{Isom}

\newcommand{\id}{\text{id}}

\newcommand{\ip}[1]{\langle #1 \rangle}
\newcommand{\set}[1]{\lbrace #1 \rbrace}
\newcommand{\reals}{\mathbb{R}}

\newcommand{\naturals}{\mathbb{N}}

\newtheorem{prop}{Proposition}[section]
\newtheorem{thm}[prop]{Theorem}
\newtheorem{lem}[prop]{Lemma}
\newtheorem{cor}[prop]{Corollary}

\newtheorem*{thm*}{Theorem}
\newtheorem*{rrthm}{Rank Rigidity Theorem}
\newtheorem*{cor*}{Corollary}

\newcommand{\cC}{\mathscr{C}}
\newcommand{\Ff}{\mathcal{F}}
\newcommand{\Jj}{\mathcal{J}}
\newcommand{\Pp}{\mathcal{P}}
\newcommand{\Rr}{\mathcal{R}}

\begin{document}

\maketitle

\noindent \textbf{Abstract.} We say that a Riemannian manifold $M$ has $\rank M \geq k$ if every geodesic in $M$ admits at least $k$ parallel Jacobi fields. The Rank Rigidity Theorem of Ballmann and Burns-Spatzier, later generalized by Eberlein-Heber, states that a complete, irreducible, simply connected Riemannian manifold $M$ of rank $k \geq 2$ (the ``higher rank'' assumption) whose isometry group $\Gamma$ satisfies the condition that the $\Gamma$-recurrent vectors are dense in $SM$ is a symmetric space of noncompact type. This includes, for example, higher rank $M$ which admit a finite volume quotient. We adapt the method of Ballmann and Eberlein-Heber to prove a generalization of this theorem where the manifold $M$ is assumed only to have no focal points. We then use this theorem to generalize to no focal points a result of Ballmann-Eberlein stating that for compact manifolds of nonpositive curvature, rank is an invariant of the fundamental group.

\section{Introduction}

In the mid-80's, building on an analysis of manifolds of nonpositive curvature of higher rank carried out by Ballmann, Brin, Eberlein, and Spatzier in \cite{BalBriEbe85} and \cite{BalBriSpa85}, Ballmann in \cite{Bal85} and Burns-Spatzier in \cite{BurSpa87-1} and \cite{BurSpa87-2} independently (and with different methods) proved their Rank Rigidity Theorem:

\begin{rrthm} 
Let $M$ be a complete, simply connected, irreducible Riemannian manifold of nonpositive curvature, rank $k \geq 2$, and curvature bounded below; suppose also $M$ admits a finite volume quotient. Then $M$ is a locally symmetric space of noncompact type.
\end{rrthm}

The theorem was later generalized by Eberlein-Heber in \cite{EbeHeb90}. They removed the curvature bound, and also generalized the condition that $M$ admit a finite volume quotient to the condition that a dense set of geodesics in $M$ be $\Gamma$-recurrent; they called this condition the ``duality condition'', for reasons not discussed here.

We aim to prove the following generalization of Eberlein and Heber's result:

\begin{rrthm} 
Let $M$ be a complete, simply connected, irreducible Riemannian manifold with no focal points and rank $k \geq 2$ with group of isometries $\Gamma$, and suppose that the $\Gamma$-recurrent vectors are dense in $M$. Then $M$ is a symmetric space of noncompact type.
\end{rrthm}

Poincar\'e recurrence implies that when $M$ admits a finite volume quotient, the $\Gamma$-recurrent vectors are dense in $M$. As a consequence we obtain the following corollary:

\begin{cor*}
Let $N$ be a complete, finite volume, irreducible Riemannian manifold with no focal points and rank $k \geq 2$; then $N$ is locally symmetric.
\end{cor*}

Since the conditions of no focal points and density of $\Gamma$-recurrent vectors pass nicely to de Rham factors, we also get a decomposition theorem:

\begin{cor*}
Let $M$ be a complete, simply connected Riemannian manifold with no focal points and with group of isometries $\Gamma$, and suppose that the $\Gamma$-recurrent vectors are dense in $SM$. Then $M$ decomposes as a Riemannian product
\[
	M = M_0 \times M_S \times M_1 \times \cdot \times M_l,
\]
where $M_0$ is a Euclidean space, $M_S$ is a symmetric space of noncompact type and higher rank, and each factor $M_i$ for $1 \leq i \leq l$ is an irreducible rank-one Riemannian manifold with no focal points.
\end{cor*}

In 1987 Ballmann and Eberlein in \cite{BalEbe87} defined the \emph{rank} of an abstract group, and used the Higher Rank Rigidity Theorem in nonpositive curvature to show that, for nonpositively curved manifolds of finite volume, rank is an invariant of the fundamental group. In our final section, we derive the necessary lemmas to show that, at least in the case the manifold is compact, their proof applies to the case of no focal points as well. Therefore we have

\begin{thm*}
Let $M$ be a complete, simply connected Riemannian manifold without focal points, and let $\Gamma$ be a discrete, cocompact subgroup of isometries of $M$ acting freely and properly on $M$. Then $\rank(\Gamma) = \rank(M)$.
\end{thm*}

As a corollary of this and the higher rank rigidity theorem, we find for instance that the locally symmetric metric is the unique Riemannian metric of no focal points on a compact locally symmetric space.

Our proof follows closely the method of Ballmann and Eberlein-Heber, as presented in Ballmann's book \cite{Bal95}. The paper is organized as follows. In section \ref{S_prelim} we recall the necessary definitions and state the results we will need on manifolds with no focal points, most of which come from a paper of O'Sullivan \cite{OSu76}. We construct a visual boundary $M(\infty)$ and derive a few of its properties. Finally, subsection \ref{sS_24} is devoted to a number of lemmas that allow us to compare the behavior of $SM$ at two possibly distant vectors whose associated geodesics are asymptotic. These lemmas rely heavily on recurrence. One of the major tools lost when passing from nonpositive curvature to no focal points is convexity of the function $t \mapsto d(\gamma(t), \sigma(t))$ for geodesics $\gamma, \sigma$, and the lemmas in subsections \ref{sS_22} and \ref{sS_24} are the key tools we use here to replace it.

We then proceed to the main proof. In section \ref{S_flats} we show that $M$ has sufficiently many $k$-flats, and in section \ref{S_angle} we investigate the structure of the visual boundary of $M$ (defined as asymptotic classes of geodesic rays). These two sections repeat for manifolds of no focal points some of the breakthrough work of Ballmann, Brin, Eberlein, and Spatzier on manifolds of nonpositive curvature, which was instrumental in the original proofs of the Rank Rigidity Theorem (see \cite{BalBriEbe85}, \cite{BalBriSpa85}). In addition, we construct in secton \ref{S_angle} a closed proper invariant subset of the visual boundary. The arguments in these two sections are generalizations of their counterparts in nonpositive curvature, originally set out in \cite{BalBriEbe85}, \cite{BalBriSpa85}, \cite{Bal85}, and \cite{EbeHeb90}.

In section \ref{S_complete} we complete the proof of the higher rank rigidity theorem via an appeal to the holonomy classification theorem of Berger-Simons. This section follows Ballmann nearly word-for-word, but the arguments are brief enough that we present them again here.

Finally, in section \ref{S_FundGroup} we prove our generalization of the theorem of Ballmann-Eberlein. We omit many proofs here that follow Ballmann-Eberlein word-for-word, or with only trivial modifications; our work here is primarily in generalizing a number of well-known lemmas from nonpositive curvature to the case of no focal points. The main new tool here is Lemma \ref{A_B}, which is used as a replacement for a type of flat strip theorem in nonpositive curvature which says that if two geodesic rays $\gamma_1$ and $\gamma_2$ meet a geodesic $\sigma$ in such a way that the sum of the interior angles is $\pi$, then $\gamma_1, \gamma_2,$ and $\sigma$ bound a flat half strip.  

The author is indebted to Ralf Spatzier for numerous conversations on the material of this paper.

\section{Preliminaries}\label{S_prelim}

\subsection{Notation}

For sections \ref{S_prelim} - \ref{S_complete} of this paper, $M$ is assumed to be a complete, simply connected, irreducible Riemannian manifold with no focal points.

We denote by $TM$ and $SM$ the tangent and unit tangent bundles of $M$, respectively, and we denote by $\pi$ the corresponding projection map. If $v$ is a unit tangent vector to a manifold $M$, we let $\gamma_v$ denote the (unique) geodesic with $\dot{\gamma}(0) = v$. 

Recall that $SM$ inherits a natural metric from $M$, the Sasake metric, as follows: we have for any $v \in TM$ a decomposition
\[
	T_vTM = T_{\pi(v)}M \oplus T_{\pi(v)}M
\]
given by the horizontal and vertical subspaces of the connection, and we therefore may give $T_vTM$ the inner product induced by this decomposition, giving a Riemannian metric on $TM$; the restriction of this metric to $SM$ is the Sasake metric.

Central to our discussion is the geodesic flow on $M$, which is the flow $g^t : SM \to SM$ defined by
\[
	g^t v = \dot{\gamma}_v(t).
\]

In sections \ref{S_prelim} - \ref{S_complete} we denote by $\Gamma$ the group of isometries of $M$. A vector $v \in SM$ is called $\Gamma$-\emph{recurrent}, or simply \emph{recurrent}, if for each neighborhood $U \subseteq SM$ of $v$ and each $T > 0$ there is $t \geq T$ and $\phi \in \Gamma$ such that $(d\phi \circ g^t) v \in U$. We assume throughout the paper that the set of $\Gamma$-recurrent vectors is dense in $SM$ (Eberlein-Heber call this the duality condition, for reasons not discussed here). This holds in particular if $M$ admits a finite volume quotient.

The \emph{rank} of $v \in SM$, or of $\gamma_v$, is the dimension of the space of parallel Jacobi fields along $\gamma_v$. The \emph{rank} of $M$ is the minimum of $\rank{v}$ over all unit tangent vectors $v$. A unit tangent $v$ is called \emph{regular} if there exists some neighborhood $U \subseteq SM$ of $v$ such that for all $w \in U$, $\rank{w} = \rank{v}$. We denote by $\Rr$ the set of regular vectors, and $\Rr_m$ the set of regular vectors of rank $m$. 

We let $k$ be the rank of $M$, and assume that $k \geq 2$. It is not difficult to see that the set of vectors of rank $\leq m$ is open for each $m$, and in particular, that the set $\Rr_k$ is open. In section \ref{S_flats}, we will construct for each $v \in SM$ a totally geodesic embedded $\reals^k$ in $M$ with $v$ in its tangent bundle; such an $\reals^k$ is called a \emph{$k$-flat}. The assumption that $k \geq 2$ will ensure that this is gives us nontrivial information about $M$.

We will also be interested in the ``behavior at $\infty$'' of geodesics on $M$. To do this we define the following two equivalence relations on vectors in $SM$ (equivalently, on geodesics in $M$): Two vectors $v, w \in SM$ are \emph{asymptotic} if $d(\gamma_v(t), \gamma_w(t))$ is bounded as $t \to \infty$, $t \geq 0$; and $v, w$ are called \emph{parallel} if $v, w$ are asymptotic and $-v, -w$ are also asymptotic. Geodesics $\gamma, \sigma$ are called asymptotic (resp. parallel) if $\dot{\gamma}(0), \dot{\sigma}(0)$ are asymptotic (resp. parallel). We develop these equivalence relations in section \ref{sS_bdry}.

\subsection{Results on no focal points}\label{sS_22}

Let $L$ be an arbitrary Riemannian manifold, $N$ a submanifold of $L$. The submanifold $N$ is said to have a \emph{focal point} at $q \in L$ if there exists a variation of geodesics $\gamma_s(t)$ with $\gamma_s(0) \in N$, $\gamma_0(a) = q$ for some $a$, $\dot{\gamma}_s(0) \perp N$ for all $s$, and $\partial_s\gamma_0(a) = 0$. Note that if $N$ is a point, then a focal point of $N$ is just a conjugate point of $N$ along some geodesic.

A Riemannian manifold $L$ is said to have \emph{no focal points} if every totally geodesic submanifold $N$ has no focal points. Equivalently, it suffices to check that for every Jacobi field $J$ along a geodesic $\gamma$ with $J(0) = 0$, $||J(t)||$ is a strictly increasing function of $t$ for $t > 0$. The results of this section hold for arbitrary Riemannian manifolds $L$ with no focal points. 

It is easy to check that Riemannian manifolds of nonpositive curvature have no focal points, and that Riemannian manifolds of no focal points have no conjugate points. Recall that for simply connected manifolds with no conjugate points, the exponential map $\exp_p : T_pL \to L$ is a diffeomorphism. There is an analog for no focal points: Recall that if $N$ is a submanifold of a Riemannian manifold $L$, we may construct the normal bundle $\nu^{\perp}N$ of $N$ in $L$, and there is an associated exponential map
\[
	\exp^{\perp}_N : \nu^{\perp}N \to L,
\]
which is just the restriction of the standard exponential map $\exp : TL \to L$. Then a totally geodesic submanifold $N$ of a Riemannian manifold has no focal points iff the map $\exp^{\perp}_N$ is a diffeomorphism. In fact, as one might expect, focal points occur exactly at the places where $d \exp^{\perp}_N$ is singular. For a reference, see for instance O'Sullivan \cite{OSu74}.


%
We now state the results we need on manifolds with no focal points. Throughout, $M$ is a Riemannian manifold with no focal points. The main reference here is O'Sullivan's paper \cite{OSu76}\footnote{Note that, as remarked by O'Sullivan himself, the relevant results in \cite{OSu76} are valid for \emph{all} manifolds with no focal points (rather than only those with a lower curvature bound), since the condition $||J(0)|| \to \infty$ for all nontrivial initially vanishing Jacobi fields $J$ is always satisfied for manifolds with no focal points, as shown by Goto \cite{Got78}.}.

First, we have the following two propositions, which often form a suitable replacement for convexity of the function $t \mapsto d(\gamma(t), \sigma(t))$ for geodesics $\gamma, \sigma$:

\begin{prop}[\cite{OSu76} \S 1 Prop 2]\label{OSu76 1}
Let $\gamma$ and $\sigma$ be distinct geodesics with $\gamma(0) = \sigma(0)$. Then for $t > 0$, both $d(\gamma(t), \sigma)$ and $d(\gamma(t), \sigma(t))$ are strictly increasing and tend to infinity as $t \to \infty$.
\end{prop}

\begin{prop}[\cite{OSu76} \S 1 Prop 4]\label{OSu76 2}
Let $\gamma$ and $\sigma$ be asymptotic geodesics; then both $d(\gamma(t), \sigma)$ and $d(\gamma(t), \sigma(t))$ are nonincreasing for $t \in \reals$.
\end{prop}

\noindent O'Sullivan also proves an existence and uniqueness result for asymptotic geodesics:

\begin{prop}[\cite{OSu76} \S 1 Prop 3]\label{OSu76 3}
Let $\gamma$ be a geodesic; then for each $p \in M$ there is a unique geodesic through $p$ and asymptotic to $\gamma$.
\end{prop}

\noindent Finally, O'Sullivan also proves a flat strip theorem (this result was also obtained, via a different method, by Eschenburg in \cite{Esc77}):

\begin{prop}[\cite{OSu76} \S 2 Thm 1]\label{OSu76 flat}
If $\gamma$ and $\sigma$ are parallel geodesics, then $\gamma$ and $\sigma$ bound a flat strip; that is, there is an isometric immersion $\phi : [0, a] \times \reals \to M$ with $\phi(0, t) = \gamma(t)$ and $\phi(a, t) = \sigma(t)$.
\end{prop}

\noindent We will also need the following result, which is due to Eberlein (\cite{Ebe73}); a proof can also be found in \cite{Esc77}. 

\begin{prop}\label{Eb bdd}
Bounded Jacobi fields are parallel.
\end{prop}

\noindent Finally, we have the following generalization of Proposition \ref{OSu76 1}:

\begin{prop}\label{OSu76 4}
Let $p \in M$, let $N$ be a totally geodesic submanifold of $M$ through $p$, and let $\gamma$ be a geodesic of $M$ with $\gamma(0) = p$. Assume $\gamma$ is not contained in $N$; then $d(\gamma(t), N)$ is strictly increasing and tends to $\infty$ as $t \to \infty$.
\end{prop}
\begin{proof}
Let $\sigma_t$ be the unique geodesic segment joining $\gamma(t)$ to $N$ and perpendicular to $N$; then (by a first variation argument) $d(\gamma(t), N) = L(\sigma_t)$, where $L(\sigma_t)$ gives the length of $\sigma_t$. Thus if $d(\gamma(t), N)$ is not strictly increasing, then we have $L'(\sigma_t) = 0$ for some $t$, and again a first variation argument establishes that then $\sigma_t$ is perpendicular to $\gamma$, which is a contradiction since $\exp : \nu^{\perp} \sigma_t \to M$ is a diffeomorphism.

This establishes that $d(\gamma(t), N)$ is strictly increasing. To show it is unbounded we argue by contradiction. Suppose
\[
	\lim_{t \to \infty} d(\gamma(t), N) = C < \infty,
\]
and choose sequences $t_n \to \infty$ and $a_n \in N$ such that $d(\gamma(t_n), N) = d(\gamma(t_n), a_n)$ and the sequence $d(\gamma(t_n), a_n)$ increases monotonically to $C$. We let $w_n$ be the unit tangent vector at $\gamma(0)$ pointing at $a_n$; by passing to a subsequence, we may assume $w_n \to w \in T_{\gamma(0)} N$. 

We claim $d(\gamma(t), \gamma_w) \leq C$ for all $t \geq 0$, contradicting Proposition \ref{OSu76 1}. Fix a time $t \geq 0$ and $\epsilon > 0$. For each $n$, there is a time $s_n$ such that
\[
	d(\gamma(t), \gamma_{w_n}) = d(\gamma(t), \gamma_{w_n}(s_n)). 
\]
The triangle inequality gives
\[
	s_n \leq t + C.
\]
Thus some subsequence of the points $\gamma_{w_n}(s_n)$ converges to a point $\gamma_w(s)$, and then clearly $d(\gamma(t), \gamma_w(s)) \leq C$, which establishes the result.
\end{proof}

\subsection{The boundary of $M$ at infinity}\label{sS_bdry}

We define for $M$ a visual boundary $M(\infty)$, the \emph{boundary of $M$ at infinity}, a topological space whose points are equivalence classes of unit speed asymptotic geodesics in $M$. If $\eta \in M(\infty)$, $v \in SM$, and $\gamma_v$ is a member of the equivalence class $\eta$, then we say $v$ (or $\gamma_v$) \emph{points at} $\eta$. 

Proposition \ref{OSu76 3} shows that for each $p \in M$ there is a natural bijection $S_p M \cong M(\infty)$ given by taking a unit tangent vector $v$ to the equivalence class of $\gamma_v$. Thus for each $p$ we obtain a topology on $M(\infty)$ from the topology on $S_p M$; in fact, these topologies (for various $p$) are all the same, which we now show.

Fix $p, q \in M$ and let $\phi : S_pM \to S_qM$ be the map given by taking $v \in S_pM$ to the unique vector $\phi(v) \in S_qM$ asymptotic to $v$. We wish to show $\phi$ is a homeomorphism, and for this it suffices to show:

\begin{lem}\label{sameTop}\label{28}
The map $\phi : S_p M \to S_qM$ is continuous.
\end{lem}
\begin{proof}
Let $v_n \in S_pM$ with $v_n \to v$, and let $w_n, w \in S_qM$ be asymptotic to $v_n, v$, respectively. We must show $w_n \to w$. Suppose otherwise; then, passing to a subsequence, we may assume $w_n \to u \neq w$. Fix $t \geq 0$. Choose $n$ such that
\begin{align*}
	d(\gamma_{w_n}(t), \gamma_u(t)) + d(\gamma_{v_n}(t), \gamma_v(t)) < d(p, q).
\end{align*}
Then
\begin{align*}
	d(\gamma_u(t), \gamma_w(t)) \leq d(\gamma_u(t)&, \gamma_{w_n}(t)) + d(\gamma_{w_n}(t), \gamma_{v_n}(t)) \\ &+ d(\gamma_{v_n}(t), \gamma_v(t)) + d(\gamma_v(t), \gamma_w(t)) < 3 d(p, q),
\end{align*}
the second and fourth terms being bounded by $d(p, q)$ by Proposition \ref{OSu76 2}. Since $t$ is arbitrary, this contradicts Proposition \ref{OSu76 1}.
\end{proof}

We call the topology on $M(\infty)$ induced by the topology on any $S_pM$ as above the \emph{visual topology}. We will be defining a second topology on $M(\infty)$ presently, so we take a moment to fix notation: If $\zeta_n, \zeta \in M(\infty)$ and we write $\zeta_n \to \zeta$, we \emph{always} mean with respect to the visual topology unless explicitly stated otherwise. 

%
%
%
%
If $\eta, \zeta \in M(\infty)$ and $p \in M$, then $\angle_p(\eta, \zeta)$ is defined to be the angle at $p$ between $v_{\eta}$ and $v_\zeta$, where $v_\eta, v_\zeta \in S_pM$ point at $\eta, \zeta$, respectively.

We now define a metric $\angle$ on $M(\infty)$, the \emph{angle metric}, by
\[
	\angle(\eta, \zeta) = \sup_{p \in M} \angle_p(\eta, \zeta).
\]
We note that the metric topology determined by $\angle$ is not in general equivalent to the visual topology. However, we do have:

\begin{prop}\label{210}
The angle metric is lower semicontinuous. That is, if $\eta_n \to \eta$ and $\zeta_n \to \zeta$ (in the visual topology), then
\[
	\angle(\eta, \zeta) \leq \liminf \angle(\eta_n, \zeta_n).
\]
\end{prop}
\begin{proof}
It suffices to show that for all $\epsilon > 0$ and all $q \in M$, we have for all but finitely many $n$
\[
	\angle_q(\eta, \zeta) - \epsilon < \angle(\eta_n, \zeta_n).
\]
Fixing $q \in M$ and $\epsilon > 0$, since $\eta_n \to \eta$ and $\zeta_n \to \zeta$, for all but finitely many $n$ we have
\[
	\angle_q(\eta, \zeta) < \angle_q(\eta_n, \zeta_n) + \epsilon,
\]
and this implies the inequality above.
\end{proof}

We also take a moment to establish a few properties of the angle metric. 

\begin{prop}
The angle metric $\angle$ is complete.
\end{prop}
\begin{proof}
For $\xi \in M(\infty)$, we denote by $\xi(p) \in S_pM$ the vector pointing at $\xi$. Let $\zeta_n$ be a $\angle$-Cauchy sequence in $M(\infty)$. Then for each $p$ the sequence $\zeta_n(p)$ is Cauchy in the metric $\angle_p$, and so has a limit $\zeta(p)$; by Lemma \ref{sameTop}, the asymptotic equivalence class of $\zeta(p)$ is independent of $p$. We denote this class by $\zeta$; it is now easy to check that $\zeta_n \to \zeta$ in the $\angle$ metric. (This follows from the fact that the sequences $\zeta_n(p)$ are Cauchy uniformly in $p$.)
\end{proof}

\begin{lem}\label{nonincr}\label{nondecr}
Let $v \in SM$ point at $\eta \in M(\infty)$, and let $\zeta \in M(\infty)$. Then $\angle_{\gamma_v(t)}(\eta, \zeta)$ is a nondecreasing function of $t$.
\end{lem}
\begin{proof}
This follows from Proposition \ref{OSu76 2} and a simple first variation argument.
\end{proof}

\subsection{Asymptotic vectors, recurrence, and the angle metric}\label{sS_24}

In this section we collect a number of technical lemmas. As a consequence we derive Corollary \ref{flatcorrect}, which says that the angle between the endpoints of recurrect vectors is measured correctly from any flat. (In nonpositive curvature, this follows from a simple triangle-comparison argument.)  

Our first lemma allows us to compare the behavior of the manifold at (possibly distant) asymptotic vectors:

\begin{lem}\label{seq}
Let $v, w \in SM$ be asymptotic. Then there exist sequences $t_n \to \infty, v_n \to v$, and $\phi_n \in \Gamma$ such that
\[
	(d\phi_n \circ g^{t_n})v_n \to w
\]
as $n \to \infty$.
\end{lem}
\begin{proof}
First assume $w$ is recurrent. Then we may choose $s_n \to \infty$ and $\phi_n \in \Gamma$ so that $(d\phi_n \circ g^{s_n})w \to w$. For each $n$ let $q_n$ be the footpoint of $g^{s_n} w$, and let $v_n$ be the vector with the same footpoint as $v$ such that the geodesic through $v_n$ intersects $q_n$ at some time $t_n$. Clearly $t_n \to \infty$.
\begin{figure}[htb]
\begin{center}
\leavevmode
\includegraphics[width=0.4\textwidth]{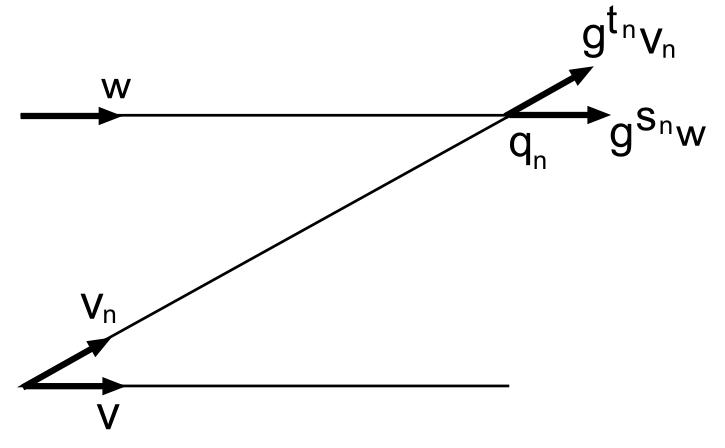}
\end{center}
\label{fig:Pic1}
\end{figure}
We now make two claims: First, that $v_n \to v$ and second, that $(d\phi_n \circ g^{t_n})v_n \to w$. Note that since $v$ and $w$ are asymptotic, Lemma \ref{nonincr} gives
\[
	\angle_{\pi(v)} (v, v_n) \leq \angle_{q_n} (g^{t_n} v_n, g^{s_n} w).
\]
So if we show that the right-hand side goes to zero, both our claims are verified. 
\begin{figure}[htb]
\begin{center}
\leavevmode
\includegraphics[width=0.4\textwidth]{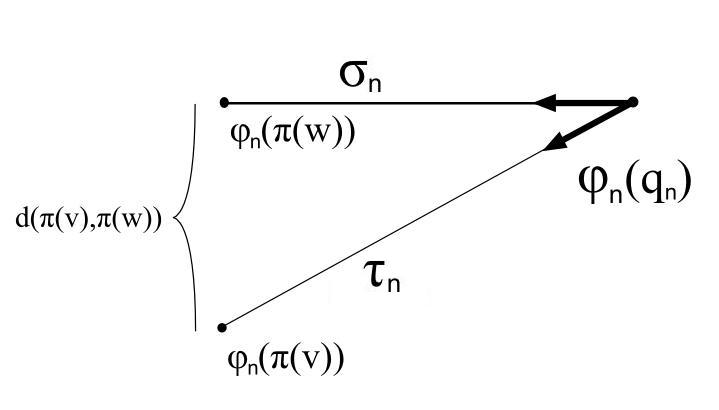}
\end{center}
\label{fig:Pic2}
\end{figure}

Consider the geodesic rays $\tau_n, \sigma_n$ through the point $\phi_n(q_n)$ satisfying
\begin{align*}
	\dot{\tau}_n(0) &= - d\phi_n(g^{t_n} v_n), & \dot{\sigma}_n(0) &= - d\phi_n(g^{s_n}w). 
\end{align*}
It suffices to show the angle between these rays goes to zero. Note $s_n, t_n \to \infty$. We claim that the distance between $\tau_n(t)$ and $\sigma_n(t)$ is bounded, independent of $n$, for $t \leq \max \set{s_n, t_n}$. To see this, first note that $|s_n - t_n| \leq d(\pi (v), \pi (w))$ by the triangle inequality. Suppose for example that $s_n \geq t_n$; then we find
\[
	d(\sigma_n(s_n), \tau_n(s_n)) \leq 2 d(\pi v, \pi w),
\]
and Proposition \ref{OSu76 1} shows that for $0 \leq t \leq s_n$,
\[
	d(\sigma_n(t), \tau_n(t)) \leq 2 d(\pi v, \pi w).
\]
The same holds if $t_n \geq s_n$. Hence for fixed $t$, for all but finitely many $n$ the above inequality holds. It follows that $\tau_n$ and $\sigma_n$ converge to asymptotic rays starting at $p$. This establishes the theorem for recurrent vectors $w$.

We now do not assume $w$ is recurrent; since recurrent vectors are dense in $SM$, we may take a sequence $w_m$ of recurrent vectors with $w_m \to w$. For each $m$, there are sequences $v_{n,m} \to v$, $t_{n,m} \to \infty$, and $\phi_{n,m} \in \Gamma$ such that
\[
	(d\phi_{n,m} \circ g^{t_{n,m}})v_{n,m} \to w_m.
\]
An appropriate ``diagonal'' argument now proves the theorem.
\end{proof}

As a corollary of the above proof we get the following:

\begin{cor}\label{43}
Let $v \in SM$ be recurrent and pointing at $\eta \in M(\infty)$; let $\zeta \in M(\infty)$. Then
\[
	\angle(\eta, \zeta) = \lim_{t \to \infty} \angle_{\gamma_v(t)}(\eta, \zeta).
\]
\end{cor}
\begin{proof}
By Lemma \ref{nonincr}, the limit exists. Let $p = \pi(v)$, and fix arbitrary $q \in M$. Since $v$ is recurrent, there exist $t_n \to \infty$ and $\phi_n \in \Gamma$ such that $(d\phi_n \circ g^{t_n})v \to v$. Let $p_n$ be the footpoint of $g^{t_n} v$, and let $\gamma_n$ be the geodesic from $q$ to $p_n$. Define
\begin{align*}
	v_n &= g^{t_n} v & &\text{and} & v_n' = \dot{\gamma}_n(p_n).
\end{align*}
\begin{figure}[htb]
\begin{center}
\leavevmode
\includegraphics[width=0.4\textwidth]{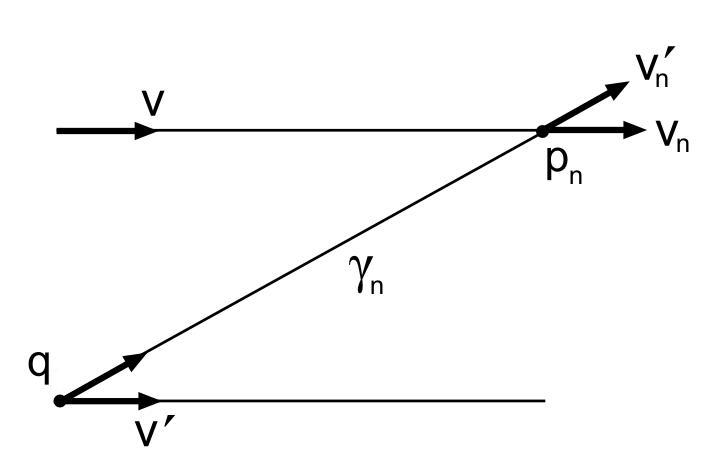}
\end{center}
\label{fig:Pic3}
\end{figure}
By the argument given in Lemma \ref{seq}, $\angle_{p_n}(v_n, v_n') \to 0$, and if we let $v' \in S_qM$ be the vector pointing at $\eta$, then $\dot{\gamma}_n(0) \to v'$. Thus
\[
	\angle_{p_n}(\zeta, v_n') \geq \angle_q(\zeta, \dot{\gamma}_n(0)) \to \angle_q(\zeta, \eta).
\]
Since $q$ was arbitrary, this proves the claim.
\end{proof}

In fact, the above corollary is true if $v$ is merely asymptotic to a recurrent vector. To prove this we will need a slight modification to Lemma \ref{seq}, which is as follows:

\begin{lem}\label{43a}
Let $w$ be recurrent and $v$ asymptotic to $w$. Then there exist sequences $w_n \to w$ and $s_n, t_n \to \infty$ such that $g^{t_n} w_n$ and $g^{s_n} v$ have the same footpoint $q_n$ for each $n$, and
\[
	\angle_{q_n}(g^{t_n}w_n, g^{s_n} v) \to 0.
\]
\end{lem}
\begin{proof}
First let $s_n \to \infty$, $\phi_n \in \Gamma$, be sequences such that
\[
	(d\phi_n \circ g^{s_n}) w \to w.
\]
Define $p = \pi(w), q = \pi(v)$, $p_n = \pi(g^{s_n}w)$, and $q_n = \pi(g^{s_n}v)$. Let $w_n$ be the unit tangent vector with footpoint $p$ such that there exists $t_n$ such that $g^{t_n} w_n$ has footpoint $q_n$. 
\begin{figure}[htb]
\begin{center}
\leavevmode
\includegraphics[width=0.4\textwidth]{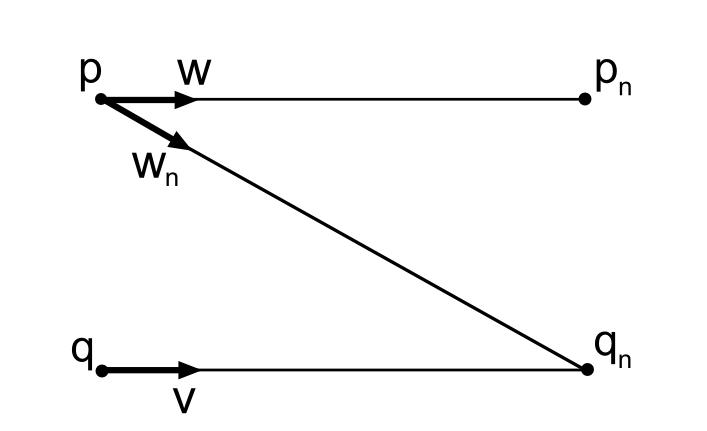}
\end{center}
\label{fig:Pic4}
\end{figure}

Note that for all $n$
\begin{align*}
	d(\phi_n(q_n), p) &\leq d(\phi_n(q_n), \phi_n(p_n)) + d(\phi_n(p_n), p) \\
					&\leq d(q_n, p_n) + K \\
					&\leq d(q, p) + K,
\end{align*}
where $K$ is some fixed constant. In particular, the points $\phi_n(q_n)$ all lie within bounded distance of $p$, and hence within some compact set. Therefore, by passing to a subsequence, we may assume we have convergence of the following three sequences:
\begin{align*}
	r_n := \phi_n(q_n) &\to r \\
	w_n' := (d\phi_n \circ g^{t_n})w_n &\to w' \\
	v_n' := (d\phi_n \circ g^{s_n})v &\to v' 
\end{align*}
for some $r, w', v'$. Then by the argument in the proof of Lemma \ref{seq}, 
\[
	d(\gamma_{-w_n'}(t), \gamma_{-v_n'}(t)) \leq 2 d(p,q)
\]
for $0 \leq t \leq \max \set{s_n, t_n}$. It follows that $(-w')$ and $(-v')$ are asymptotic; since both have footpoint $r$, we see $w' = v'$. This gives the lemma.
\end{proof}

We can now prove our previous claim:

\begin{prop}
Let $w \in SM$ be recurrent, $v$ asymptotic to $w$. Say $v$ and $w$ both point at $\eta \in M(\infty)$. Then for all $\zeta \in M(\infty)$
\[
	\angle(\eta, \zeta) = \lim_{t \to \infty} \angle_{\gamma_v(t)}(\eta, \zeta).
\]
\end{prop}
\begin{proof}
Fix $\epsilon > 0$. By Lemma \ref{43}, there exists a $T$ such that
\[
	\angle_{\gamma_w(T)}(\eta, \zeta) \geq \angle(\eta, \zeta) - \epsilon.
\]
We write $w' = g^Tw$ and note that $w'$ is also recurrent and asymptotic to $v$. Let $p$ be the footpoint of $w'$. Choose by Lemma \ref{43a} sequences $w_n \to w'$ and $s_n, t_n \to \infty$ such that
\begin{equation*}
	\angle_{\gamma_{v}(s_n)}(g^{t_n}w_n, g^{s_n}v) \to 0. 		\tag{(*)}
\end{equation*}
To fix notation, let $w_n$ point at $\eta_n$. Then for large $n$
\begin{align*}
	\angle_{\gamma_v(s_n)}(\eta, \zeta) &\geq \angle_{\gamma_v(s_n)}(\eta_n, \zeta) - \epsilon & &\text{by (*)} \\
		&\geq \angle_{p}(\eta_n, \zeta) - \epsilon & &\text{by Lemma \ref{nonincr}}\\
		&\geq \angle_p(\eta, \zeta) - 2\epsilon & &\text{by definition of the visual topology}\\
		&\geq \angle(\eta, \zeta) - 3\epsilon & &\text{by construction of $w'$}.
\end{align*}
\end{proof}

The key corollary of these results is:

\begin{cor}\label{flatcorrect}\label{C}
Let $\eta$ be the endpoint of a recurrent vector $w$. Let $F$ be a flat at $q \in M$, and $v, v' \in S_qF$ with $v$ pointing at $\eta$. Say $v'$ points at $\zeta$; then
\[
	\angle(\eta, \zeta) = \angle_q(\eta, \zeta).
\]
\end{cor}

In the next section we will establish the existence of plenty of flats; in section \ref{S_angle}, this corollary will be one of our primary tools when we analyze the structure of the angle metric on $M(\infty)$.

\section{Construction of flats}\label{S_flats}

We repeat our standing assumption that $M$ is a complete, simply connected, irreducible Riemannian manifold of higher rank and no focal points. 

For a vector $v \in SM$, we let $\Pp(v) \subseteq SM$ be the set of vectors parallel to $v$, and we let $P_v$ be the image of $\Pp(v)$ under the projection map $\pi : SM \to M$. Thus, $p \in P_v$ iff there is a unit tangent vector $w \in T_pM$ parallel to $v$. Our goal in this section will be to show that if $v$ is a regular vector of rank $m$, that is, $v \in \Rr_m$, then the set $P_v$ is an $m$-flat (a totally geodesic isometrically embedded copy of $\reals^m$). To this end, we will first show that $\Pp(v)$ is a smooth submanifold of $\Rr_m$. 

We begin by recalling that if $v \in SM$, there is a natural identification of $T_vTM$ with the space of Jacobi fields along $\gamma_v$. In particular, the connection gives a decomposition of $T_vTM$ into horizontal and vertical subspaces
\[
	T_vTM \cong T_{\pi(v)}M \oplus T_{\pi(v)} M,
\]
and we may identify an element $(x, y)$ in the latter space with the unique Jacobi field $J$ along $\gamma_v$ satisfying $J(0) = x, J'(0) = y$. Under this identification, $T_vSM$ is identified with the space of Jacobi fields $J$ such that $J'(t)$ is orthogonal to $\dot{\gamma}_v(t)$ for all $t$. 

Define a distribution $\Ff$ on the bundle $TSM \to SM$ by letting $\Ff(v) \subseteq T_v SM$ be the space of parallel Jacobi fields along $\gamma_v$. The plan is to show that $\Ff$ is smooth and integrable on $\Rr_m$, and its integral manifold is exactly $\Pp(v)$. We note first that $\Ff$ is continuous on $\Rr_m$, since the limit of a sequence of parallel Jacobi fields is a parallel Jacobi field, and the dimension of $\Ff$ is constant on $\Rr_m$.

\begin{lem}\label{F smooth}
$\Ff$ is smooth as a distribution on $\Rr_m$.
\end{lem}
\begin{proof}
For $w \in SM$ let $\Jj_0(w)$ denote the space of Jacobi fields $J$ along $\gamma_w$ satisfying $J'(0) = 0$. For each $w \in \Rr_m$ and each $t > 0$, consider the quadratic form $Q_t^w$ on $\Jj_0(w)$ defined by
\[
	Q_t^w(X, Y) = \int_{-t}^t \ip{R(X, \dot{\gamma}_w)\dot{\gamma}_w, R(Y, \dot{\gamma}_w)\dot{\gamma}_w} dt.
\]
Since a Jacobi field $J$ satisfying $J'(0) = 0$ is parallel iff $R(J, \dot{\gamma}_w)\dot{\gamma}_w = 0$ for all $t$, we see that $\Ff(w)$ is exactly the intersections of the nullspaces of $Q_t^w$ over all $t > 0$. In fact, since the nullspace of $Q_t^w$ is contained in the nullspace of $Q_s^w$ for $s < t$, there is some $T$ such that $\Ff(w)$ is exactly the nullspace of $Q^w_T$. We define $T(w)$ to be the infimum of such $T$; then $\Ff(w)$ is exactly the nullspace of $Q^w_{T(w)}$.

We claim that the map $w \mapsto T(w)$ is upper semicontinuous on $\Rr_m$. We prove this by contradiction. Suppose $w_n \to w$ with $w_n \in \Rr_m$, and suppose that $\limsup T(w_n) > T(w)$. Passing to a subsequence of the $w_n$, we may find for each $n$ a Jacobi field $Y_n$ along $\gamma_{w_n}$ satisfying $Y_n'(0) = 0$ and such that $Y_n$ is parallel along the segment of $\gamma_{w_n}$ from $-T(w)$ to $T(w)$, but not along the segment from $-T(w_n)$ to $T(w_n)$.

We project $Y_n$ onto the orthogonal complement to $\Ff(w_n)$, and then normalize so that $||Y_n(0)|| = 1$. Clearly $Y_n$ retains the properties stated above. Then, passing to a further subsequence, we may assume $Y_n \to Y$ for some Jacobi field $Y$ along $\gamma_w$. Then $Y$ is parallel along the segment of $\gamma_w$ from $-T(w)$ to $T(w)$. However, since $\Ff$ is continuous and $Y_n$ is bounded away from $\Ff$, $Y$ cannot be parallel along $\gamma_w$. This contradicts the choice of $T(w)$, and establishes our claim that $w \mapsto T(w)$ is upper semicontinuous.

To complete the proof, fix $w \in \Rr_m$ and choose an open neighborhood $U \subseteq \Rr_m$ of $w$ such that $\overline{U}$ is compact and contained in $\Rr_m$. Since $T(w)$ is upper semicontinuous it is bounded above by some constant $T_0$ on $U$. But then the nullspace of the form $Q^u_{T_0}$ is exactly $\Ff(u)$ for all $u \in U$; since $Q^u_{T_0}$ depends smoothly on $u$ and its nullspace is $m$-dimensional on $U$, its nullspace, and hence $\Ff$, is smooth on $U$. 
\end{proof}

Our goal is to show that $\Ff$ is in fact integrable on $\Rr_m$; the integral manifold through $v \in \Rr_m$ will turn out to then be $\Pp(v)$, the set of vectors parallel to $v$. To apply the Frobenius theorem, we will use the following lemma, which states that curves tangent to $\Ff$ are exactly those curves consisting of parallel vectors:

\begin{lem}\label{F tan}
Let $\sigma : (-\epsilon, \epsilon) \to \Rr_m$ be a curve in $\Rr_m$; then $\sigma$ is tangent to $\Ff$ (for all $t$) iff for any $s, t \in (-\epsilon, \epsilon)$, the vectors $\sigma(s)$ and $\sigma(t)$ are parallel.
\end{lem}
\begin{proof}
First let $\sigma : (-\epsilon, \epsilon) \to \Rr_m$ be a curve tangent to $\Ff$. Consider the geodesic variation $\Phi : (-\epsilon, \epsilon) \times (-\infty, \infty) \to M$ determined by $\sigma$:
\[
	\Phi(s, t) = \gamma_{\sigma(s)}(t).
\]
By construction and our identification of Jacobi fields with elements of $TTM$, we see that the variation field of $\Phi$ along the curve $\gamma_{\sigma(s)}$ is a Jacobi field corresponding exactly to the element $\dot{\sigma}(s) \in T_{\sigma(s)}TM$, and, by definition of $\Ff$, is therefore parallel. The curves $s \mapsto \Phi(s, t_0)$ are therefore all the same length $L$ (as $t_0$ varies), and thus for any $s, s'$ and all $t$
\[
	d( \gamma_{\sigma(s)}(t), \gamma_{\sigma(s')}(t) ) \leq L.
\]
Thus (by definition) $\sigma(s)$ and $\sigma(s')$ are parallel.

Conversely, let $\sigma : (-\epsilon, \epsilon) \to \Rr_m$ consist of parallel vectors and construct the variation $\Phi$ as before. We wish to show that the variation field $J(t)$ of $\Phi$ along $\gamma_{\sigma(0)}$ is parallel along $\gamma_{\sigma(0)}$, and for this it suffices, by Proposition \ref{Eb bdd}, to show that it is bounded.

Our assumption is that the geodesics $\gamma_s(t) = \Gamma(s, t)$ are all parallel (for varying $s$), and thus for any $s$ the function $d(\gamma_0(t), \gamma_s(t))$ is constant (by Proposition \ref{OSu76 2}). It follows that $||J(t)|| = ||J(0)||$ for all $t$, which gives the desired bound.
\end{proof}

Any curve $\sigma : (-\epsilon, \epsilon) \to \Rr_m$ defines a vector field along the curve (in $M$) $\pi \circ \sigma$ in the obvious way. It follows from the above lemma (and the symmetry $D_t \partial_s \Phi = D_s \partial_t \Phi)$ for variations $\Phi$) that if $\sigma$ is a curve in $\Rr_m$ such that $\sigma(t)$ and $\sigma(s)$ are parallel for any $t, s$, then the associated vector field along $\pi \circ \sigma$ is a parallel vector field along $\pi \circ \sigma$. 

We also require the following observation. Suppose that $p, q \in M$ are connected by a minimizing geodesic segment $\gamma : [0, a] \to M$, and let $v \in T_pM$. Then the curve $\sigma : [0, a] \to SM$ such that $\sigma(t)$ is the parallel transport of $v$ along $\gamma$ to $\gamma(t)$ is a minimizing geodesic in the Sasake metric. It follows from this and the flat strip theorem that if $v, w$ are parallel and connected by a unique minimizing geodesic in $SM$, then this geodesic is given by parallel transport along the unique geodesic from $\pi(v)$ to $\pi(w)$ in $M$ and is everywhere tangent to $\Ff(v)$. 

\begin{lem}
$\Ff$ is integrable as a distribution on $\Rr_m$, and, if $v \in \Rr_m$, then the integral manifold through $v$ is an open subset of $\Pp(v)$.
\end{lem}
\begin{proof}
To show integrability, we wish to show that $[X, Y]$ is tangent to $\Ff$ for vector fields $X, Y$ tangent to $\Ff$. If $\phi_t, \psi_s$ are the flows of $X, Y$, respectively, then $[X, Y]_v = \dot{\sigma}(0)$, where $\sigma$ is the curve
\[
	\sigma(t) = \psi_{-\sqrt{t}} \phi_{-\sqrt{t}} \psi_{\sqrt{t}} \phi_{\sqrt{t}}(v).
\]
From Lemma \ref{F tan} we see that $\sigma(0)$ and $\sigma(t)$ are parallel for all small $t$, which, by the other implication in Lemma \ref{F tan}, shows that $[X, Y]_v \in \Ff(v)$ as desired. So $\Ff$ is integrable.

Now fix $v \in \Rr_m$ and let $Q$ be the integral manifold of $\Ff$ through $v$. By Lemma \ref{F tan}, $Q \subseteq \Pp(v)$. Let $w \in Q$ and let $U$ be a normal neighborhood of $w$ contained in $\Rr_m$ (in the Sasake metric); to complete the proof it suffices to show that $U \cap \Pp(v) \subseteq Q$. Take $u \in U \cap \Pp(v)$. Then (by the observation preceding the lemma) the $SM$-geodesic from $w$ to $u$ is contained in $\Rr_m$ and consists of vectors parallel to $w$, and hence to $v$. Thus $u \in Q$.
\end{proof}

For $v \in \Rr_m$ it now follows that $\Pp(v) \cap \Rr_m$ is a smooth $m$-dimensional submanifold of $\Rr_m$, and since the $SM$-geodesic between nearby points in $\Rr_m$ is contained in $\Pp(v)$, we see that $\Pp(v)$ is totally geodesic. 

Consider the projection map $\pi : \Pp(v) \to P_v$; its differential $d\pi$ takes $(X, 0) \in \Ff(v) \subseteq T_vSM$ to $X \in T_{\pi(v)}M$. It follows that $P_v$ is a smooth $m$-dimensional submanifold of $M$ near those points $p \in M$ which are footpoints of vectors $w \in \Rr_m$ (and that $\pi$ gives a local diffeomorphism of $\Pp(v)$ and $P_v$ near such vectors $w$). We would like to extend this conclusion to the whole of $P_v$, and for this we will make use of Lemma \ref{seq}.

\begin{prop}
For every $v \in \Rr_m$, the set $P_v$ is a convex $m$-dimensional smooth submanifold of $M$.
\end{prop}
\begin{proof}
Fix $v \in \Rr_m$. The flat strip theorem shows that $P_v$ contains the $M$-geodesic between any two of its points, i.e., is convex. So we must show that $P_v$ is an $m$-dimensional smooth submanifold of $M$.

For $u \in \Rr_m$, we let $C_{\epsilon}(u) \subseteq T_{\pi(u)}M$ be the intersection of the subspace $T_{\pi(u)}P_u$ with the $\epsilon$-ball in $T_{\pi(u)}M$. Since $\Ff$ is smooth and integrable the foliation $\Pp$ is continuous with smooth leaves on $\Rr_m$; it follows that we may fix $\epsilon > 0$ and a neighborhood $U \subseteq \Rr_m$ of $v$ such that for $u \in U$,
\[
	\exp_{\pi(u)} C_{\epsilon}(u) = P_u \cap B_{\epsilon}(\pi(u)),
\]
where for $p \in M$ we denote by $B_p(\epsilon)$ the ball of radius $\epsilon$ about $p$ in $M$.

By the flat strip theorem, the above equation is preserved under the geodesic flow; that is, for all $t$ and all $u \in U$ we have
\[
	\exp_{\pi(g^t u)} C_{\epsilon}(g^t u) = P_{g^t u} \cap B_{\epsilon}(\pi(g^t u)).
\]
This equation is also clearly also preserved under isometries.

Now fix $w \in \Pp(v)$; our goal is to show that $P_v$ is smooth near $\pi(w)$. Choose by Lemma \ref{seq} sequences $v_n \to v, t_n \to \infty$, and $\phi_n \in \Gamma$ such that $(d\phi_n \circ g^{t_n})v_n \to w$. We may assume $v_n \in U$ for all $n$. For ease of notation, let $w_n = (d\phi_n \circ g^{t_n})v_n$; then for all $n$ we have $w_n \in \Rr_m$, and
\[
	\exp_{\pi(w_n)} C_{\epsilon}(w_n) = P_{w_n} \cap B_{\epsilon}(\pi(w_n)).
\]

By passing to a subsequence if necessary, we may assume the sequence of $m$-dimensional subspaces $d\pi( \Ff(w_n))$ converges to a subspace $W \subseteq T_{\pi(w)}M$. Denote by $W_\epsilon$ the $\epsilon$-ball in $W$. Then taking limits in the above equation we see that
\[
	\exp_{\pi(w)} W_\epsilon \subseteq P_w = P_v.
\]

To complete the proof, we note that since $P_v$ is convex (globally) and $m$-dimensional near $v$, $P_v$ cannot contain an $(m+1)$-ball, for then convexity would show that it contains an $(m+1)$-ball near $v$. Thus if $U' \subseteq B_{\epsilon}(w)$ is a normal neighborhood of $w$, we must have
\[
	P_w \cap U' = \exp_{\pi(w)}(W_\epsilon) \cap U',
\]
which shows that $P_v$ is a smooth $m$-dimensional submanifold of $M$ near $w$ and completes the proof.
\end{proof}

\begin{prop}
For every $v \in \Rr_m$, the set $P_v$ is an $m$-flat.
\end{prop}
\begin{proof}
Let $p = \pi(v)$. Choose a neighborhood $U$ of $v$ in $\Rr_m \cap T_{p}P_v$ such that for each $w \in U$, the geodesic $\gamma_w$ admits no nonzero parallel Jacobi field orthogonal to $P_v$. We claim $P_w = P_v$ for all $w \in U$.

To see this, recall that $T_p P_w$ is the span of $Y(0)$ for parallel Jacobi fields $Y(t)$ along $\gamma_w$. If $Y$ is such a field, then the component $Y^{\perp}$ of $Y$ orthogonal to $P_v$ is a bounded Jacobi field along $\gamma_w$, hence parallel, and therefore zero; it follows that $T_p P_w = T_p P_v$. Since $P_v$ and $P_w$ are totally geodesic, this gives $P_v = P_w$ as claimed.

But now take $m$ linearly independent vectors in $U$; by the above we may extend these to $m$ independent and everywhere parallel vector fields on $P_v$. Hence $P_v$ is flat.
\end{proof}

\begin{cor}\label{flats exist}
For every $v \in SM$, there exists a $k$-flat $F$ with $v \in S_{\pi(v)}F$.
\end{cor}
\begin{proof}
Let $v_n$ be a sequence of regular vectors with $v_n \to v$. Passing to a subsequence if necessary, we may assume there is some $m \geq k$ such that $v_n \in \Rr_m$ for all $n$. For each $n$ let $W_n$ be the $m$-dimensional subspace of $T_{\pi(v_n)}M$ such that $\exp(W_n) = P_{v_n}$. Passing to a further subsequence, we may assume $W_n \to W$, where $W$ is an $m$-dimensional subspace of $T_{\pi(v)}M$, and it is not difficult to see that $\exp W$ is an $m$-flat through $v$.
\end{proof}

\section{The angle lemma, and an invariant set at $\infty$}\label{S_angle}

The goal of the present section is to establish that $M(\infty)$ has a nonempty, proper, closed, $\Gamma$-invariant subset $X$. Our strategy is that of Ballmann \cite{Bal95} and Eberlein-Heber \cite{EbeHeb90}. In section \ref{S_complete} we will use this set to define a nonconstant function $f$ on $SM$, the ``angle from $X$'' function, which will be holonomy invariant, and this will show that the holonomy group acts nontransitively on $M$. 

Roughly speaking $X$ will be the set of endpoints of vectors of maximum singularity in $SM$; more precisely, in the language of symmetric spaces, it will turn out that $X$ is the set of vectors which lie on the one-dimensional faces of Weyl chambers. To ``pick out'' these vectors from our manifold $M$, we will use the following characterization: For each $\zeta \in M(\infty)$, we may look at the longest curve $\zeta(t) : [0, \alpha(\zeta)] \to M(\infty)$ starting at $\zeta$ and such that 
\[
	\angle_q(\zeta(t), \zeta(s)) = |t - s|
\]
for every point $q \in M$; then $\zeta$ is ``maximally singular'' (i.e., $\zeta \in X$) if $\alpha(\zeta)$ (the length of the longest such curve) is as large as possible. One may check that in the case of a symmetric space this indeed picks out the one-dimensional faces of the Weyl chambers.

To show that the set so defined is proper, we will show that it contains no regular recurrent vectors; this is accomplished by demonstrating that every such path with endpoint at a regular recurrent vector extends to a longer such path in a neighborhood of that vector. For this we will need a technical lemma that appears here as Corollary \ref{47}.

We begin with the following lemma, which shows that regular geodesics have to ``bend'' uniformly away from flats:

\begin{lem}\label{42}
Let $k = \rank M$, $v \in \Rr_k$, and let $\zeta = \gamma_v(-\infty), \eta = \gamma_v(\infty)$. Then there exists an $\epsilon > 0$ such that if $F$ is a $k$-flat in $M$ with $d(\pi(v), F) = 1$, then
\[
	\angle(\zeta, F(\infty)) + \angle(\eta, F(\infty)) \geq \epsilon.
\]
\end{lem}
\begin{proof}
By contradiction. If the above inequality does not hold for any $\epsilon$, we can find a sequence $F_n$ of $k$-flats satisfying $d(\pi(v), F_n) = 1$ and 
\[
	\angle(\zeta, F_n(\infty)) + \angle(\eta, F_n(\infty)) < 1/n.
\]
By passing to a subsequence, we may assume $F_n \to F$ for some flat $F$ satisfying $d(\pi(v), F) = 1$, and $\eta, \zeta \in F(\infty)$. In particular, $F$ is foliated by geodesics parallel to $v$, so that $\Pp(v)$ is at least $(k+1)$-dimensional, contradicting $v \in \Rr_k$.
\end{proof}

This allows us to prove the following ``Angle Lemma'':

\begin{lem}\label{45}
Let $k = \rank M$. Let $v \in \Rr_k$ be recurrent and suppose $v$ points at $\eta_0 \in M(\infty)$. Then there exists $A > 0$ such that for all $\alpha \leq A$, if $\eta(t)$ is a path
\[
	\eta(t) : [0, \alpha] \to M(\infty)
\]
satisfying $\eta(0) = \eta_0$ and
\[
	\angle(\eta(t), \eta_0) = t
\]
for all $t \in [0, \alpha]$, then $\eta(t) \in P_v(\infty)$ for all $t \in [0, \alpha]$.
\end{lem}
\begin{proof}
Let $p = \pi(v)$ be the footpoint of $v$ and let $\xi = \gamma_v(-\infty)$. By Lemma \ref{42} we may fix $\epsilon > 0$ such that if $F$ is a $k$-flat with $d(p, F) = 1$, then
\[
	\angle(\xi, F(\infty)) + \angle(\eta_0, F(\infty)) > \epsilon.
\]
Choose $\delta > 0$ such that if $w \in S_pM$ with $\angle_p(v, w) < \delta$ then $w \in \Rr_k$, and set $A = \tfrac{1}{2} \min \set{\delta, \epsilon}$. Fix $\alpha \leq A$.

For the sake of contradiction, suppose there exists a path $\eta(t) : [0, \alpha] \to M(\infty)$ as above, but for some time $a \leq \alpha$
\[
	\eta(a) \notin P_v(\infty).
\]
For $0 \leq s \leq a$, let $\eta_p(s) \in S_pM$ be the vector pointing at $\eta(s)$; since $\alpha < \delta$, we have $\eta_p(s) \in \Rr_k$. Fixing more notation, let $w = \eta_p(a)$.

We claim $\eta_0 \notin P_w(\infty)$. To see this, suppose $\eta_0 \in P_w(\infty)$; then by convexity $P_w(\infty)$ contains the geodesic $\gamma_v$, and since $\gamma_v$ is contained in a unique $k$-flat, we conclude $P_w = P_v$, which contradicts our assumption that $\eta(a) \notin P_v(\infty)$.

It follows from Proposition \ref{OSu76 4} that
\[
	d(\gamma_v(t), P_w) \to \infty \text{ as } t \to \infty.
\]
Since $v$ is recurrent, we may fix $t_n \to \infty$ and $\phi_n \in \Gamma$ such that the sequence $v_n = (d\phi_n \circ g^{t_n})v$ converges to $v$. By the above we may also assume $d(\gamma_v(t_n), P_w) \geq 1$ for all $n$. Then, since $P_u$ depends continuously on $u \in \Rr_k$, there exists $s_n \in [0, a]$ such that
\[
	d(\gamma_v(t_n), P_{\eta_p(s_n)}) = 1.
\]
\begin{figure}[htb]
\begin{center}
\leavevmode
\includegraphics[width=0.4\textwidth]{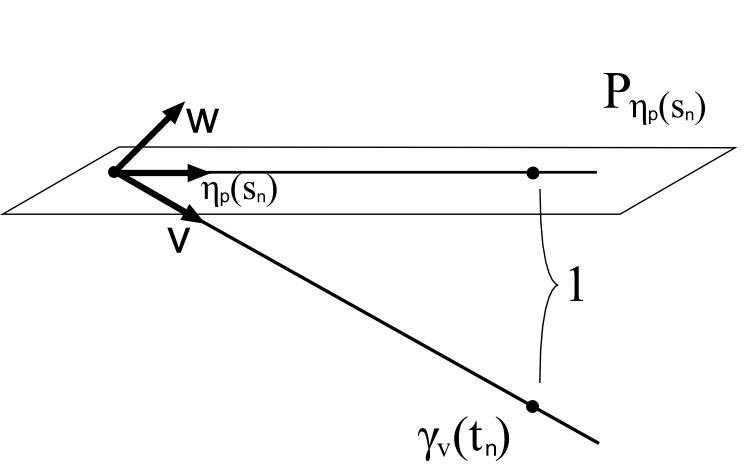}
\end{center}
\label{fig:Pic5}
\end{figure}

We define a sequence of flats $F_n$ by
\[
	F_n = \phi_n(P_{\eta_p(s_n)}).
\]

Notice that $F_n$ is indeed a flat, that $d(F_n, p) \to 1$, and that the geodesic $\gamma_{-v_n}$ intersects $F_n$ at time $t_n$. By Proposition \ref{OSu76 4}, we have
\[
	d(\gamma_{-v_n}(t), F_n) \leq 1 \text{ for } 0 \leq t \leq t_n.
\]
By passing to a subsequence, we may assume $F_n \to F$ for some $k$-flat $F$ with $d(F, p) = 1$, and taking the limit of the above inequality, we see that $\gamma_{-v}(\infty) \in F(\infty)$. Thus Lemma \ref{42} guarantees
\[
	\angle(\eta_0, F(\infty)) \geq \epsilon.
\]

On the other hand, consider the sequence $\eta(s_n)$. By passing to a further subsequence, we may assume $\phi_n(\eta(s_n)) \to \mu$; since (by definition) $\phi_n(\eta(s_n)) \in F_n(\infty)$, we have $\mu \in F(\infty)$. Then
\begin{align*}
	\epsilon &\leq \angle(\eta_0, F(\infty)) \leq \angle(\eta_0, \mu) \\
		&\leq \liminf_{n \to \infty} \angle(\phi_n(\eta_0), \phi_n(\eta(s_n)))) \\
		&= \liminf_{n \to \infty} \angle(\eta_0, \eta(s_n)) \leq a \leq \alpha \leq \frac{\epsilon}{2}, 
\end{align*}
where the inequality on the second line follows from Proposition \ref{210}. This is the desired contradiction.
\end{proof}

As we did in section \ref{sS_24}, we wish to extend this result not just to the $k$-flat $F$ containing the regular recurrent vector $v$, but to every $k$-flat containing $\eta_0$ as an endpoint at $\infty$.

\begin{prop}\label{46}
Let $v \in \Rr_k$ be recurrent and point at $\eta_0$, let $A$ be as in Lemma \ref{45} above, and let $\alpha \leq A$. Let $F$ be a $k$-flat with $\eta_0 \in F(\infty)$, and suppose there exists a path
\[
	\eta(t) : [0, \alpha] \to M(\infty)
\]
with $\eta(0) = \eta_0$ and
\[
	\angle(\eta(t), \eta_0) = t \text{ for all } t \in [0, \alpha].
\]
Then $\eta(t) \in F(\infty)$ for all $t \in [0, \alpha]$.
\end{prop}
\begin{proof} 
Fix $q \in F$, and let $\eta_q \in S_qF$ point at $\eta_0$. Let $p = \pi(v)$, and let $\phi : S_qF \to S_pM$ be the map such that $w$ and $\phi(w)$ are asymptotic. Denote by $B^F_\alpha(\eta_q)$ the restriction to $F$ of the closed $\alpha$-ball in the $\angle_q$-metric about $\eta_q$, and, similarly, denote by $B^{P_v}_\alpha(v)$ the restriction to $P_v$ of the closed $\alpha$-ball in the $\angle_p$-metric about $v$. We will show that $\phi$ gives a homeomorphism $B^F_\alpha(\eta_q) \to B^{P_v}_\alpha(v)$.

We first take a moment to note why this proves the proposition. We let $\eta_p(t) \in S_pM$ be the vector pointing at $\eta(t)$. Lemma \ref{45} tells us that $\eta_p(t) \in B^{P_v}_\alpha(v)$ for $t \in [0, \alpha]$. Then since $\phi^{-1}$ takes $B^{P_v}_\alpha(v)$ into $B^F_\alpha(\eta_q)$, we see that $\eta(t) \in F(\infty)$ for such $t$.

So we've left to show $\phi$ gives such a homeomorphism. First, let's see that $\phi$ takes $B^F_\alpha(\eta_q)$ into $B^{P_v}_\alpha(v)$. Let $w \in B^F_\alpha(\eta_q)$ and let
\[
	\sigma : [0, \alpha] \to B^F_\alpha(\eta_q)
\]
be the $\angle_q$-geodesic with $\sigma(0) = \eta_q$ and $\sigma(a) = w$ for some time $a$. Let
\[
	\tilde{\sigma} : [0, \alpha] \to M(\infty)
\]
be the path obtained by projecting $\sigma$ to $M(\infty)$. Then Corollary \ref{flatcorrect} guarantees that $\tilde{\sigma}$ satisfies the hypotheses of Lemma \ref{45}, and so we conclude that $\tilde{\sigma}(t) \in P_v(\infty)$ for all $t$, from which it follows that $\phi$ maps $B^F_\alpha(\eta_q)$ into $B^{P_v}_\alpha(v)$ as claimed.

Now, note that for all $w \in B^F_\alpha(\eta_q)$ we have
\[
	\angle_q(w, \eta_q) = \angle_p(\phi(w), v),
\]
again by Corollary \ref{flatcorrect}. Therefore for each $r \in [0, \alpha]$, $\phi$ gives an injective continuous map of the sphere of radius $r$ in $B^F_\alpha(\eta_q)$ to the sphere of radius $r$ in $B^{P_v}_\alpha(v)$; but any injective continuous map of spheres is a homeomorphism, and it follows that $\phi$ gives a homeomorphism of $B^F_\alpha(\eta_q)$ and $B^{P_v}_\alpha(v)$ as claimed.
\end{proof}

\begin{cor}\label{47}
Let $v \in \Rr_k$ be recurrent and point at $\eta_0$, let $A$ be as in Lemma \ref{45}, and let $\alpha \leq A$. Suppose we have a path
\[ 
	\eta(t) : [-\alpha, \alpha] \to M(\infty)
\]
with $\eta(0) = \eta_0$ and
\[
	\angle(\eta(t), \eta(0)) = t \text{ for all } t.
\]
Then for all $q \in M$ and all $r, s \in [-\alpha, \alpha]$
\[
	\angle_q(\eta(r), \eta(s)) = \angle(\eta(r), \eta(s)).
\]
\end{cor}
\begin{proof}
Choose two points $q_1, q_2 \in M$. Then by Corollary \ref{flats exist} there are $k$-flats $F_1, F_2$ through $q_1, q_2$, respectively, with $\eta_0 \in F_1(\infty) \cap F_2(\infty)$. By Corollary \ref{46}, the path $\eta(t)$ lifts to paths $\eta_1(t) \subseteq S_{q_1}F_1$, $\eta_2(t) \subseteq S_{q_2}F_2$. 

Fix $r, s \in [-\alpha, \alpha]$. Then for $i \in \set{1, 2}$ we have
\[
	d(\gamma_{\eta_i(r)}(t), \gamma_{\eta_i(s)}(t)) = 2t\sin\Big(\tfrac{1}{2} \big( \angle_{q_i}(\eta(r), \eta(s))\big) \Big).
\]
Since $d(\gamma_{\eta_1(r)}(t), \gamma_{\eta_2(r)}(t)$ and $d(\gamma_{\eta_1(s)}(t), \gamma_{\eta_2(s)}(t))$ are both bounded as $t \to \infty$, we must have $\angle_{q_1}(\eta_1(r), \eta_1(s)) = \angle_{q_2}(\eta_2(r), \eta_2(s))$. Thus $\angle_q(\eta(r), \eta(s))$ is independent of $q \in M$, which gives the result.
\end{proof}

\begin{prop}\label{48}
$M(\infty)$ contains a nonempty proper closed $\Gamma$-invariant subset.
\end{prop}
\begin{proof}
For each $\delta > 0$ define $X_\delta \subseteq M(\infty)$ to be the set of all $\xi \in M(\infty)$ such that there exists a path
\[
	\xi(t) : [0, \delta] \to M(\infty)
\]
with $\xi(0) = \xi$ and
\[
	\angle_q(\xi(t), \xi(s)) = |t - s|
\]
for all $t, s \in [0, s]$, and all $q \in M$.

Obviously $X_\delta$ is $\Gamma$-invariant. We claim it is closed. To this end, let $\xi_n \in X_\delta$ with $\xi_n \to \xi$, and choose associated paths
\[
	\xi_n(t) : [0, \delta] \to M(\infty).
\]
By Arzela-Ascoli, some subsequence of these paths converges (pointwise, say) to a path $\xi(t)$, and this path satisfies
\[
	\angle_q(\xi(t), \xi(s)) = \lim_{n \to \infty} \angle_q(\xi_n(t), \xi_n(s)) = |t - s|,
\]
so $\xi \in X_{\delta}$. Thus $X_\delta$ is closed; it follows that $X_\delta$ is compact.

We claim now that $X_\delta$ is nonempty for some $\delta > 0$. To see this choose a recurrent vector $v \in \Rr_k$, and say $v$ points at $\eta$. Let $A$ be as in Lemma \ref{45}, and let
\[
	\eta(t) : [0, A] \to M(\infty)
\]
be the projection to $M(\infty)$ of any geodesic segment of length $A$ starting at $v$ in $S_pP_v$. Then by Corollary \ref{C}, for all $t \in [0, A]$
\[
	\angle(\eta(t), \eta) = \angle_p(\eta(t), \eta) = t.
\]
Thus by Corollary \ref{47}, $\angle_q(\eta(s), \eta(t))$ is independent of $q \in M$, and so in particular for any such $q$
\[
	\angle_q(\eta(s), \eta(t)) = \angle_p(\eta(s), \eta(t)) = |t - s|.
\]
So $v \in X_A$.

A few remarks about the relationships between the various $X_\delta$ are necessary before we proceed. First of all, notice that if $\delta_1 < \delta_2$ then $X_{\delta_2} \subseteq X_{\delta_1}$. Furthermore, for any $\delta$, we claim that $\xi \in X_{\delta}$ iff $\xi \in X_{\epsilon}$ for all $\epsilon < \delta$. One direction is clear. To see the other, suppose $\xi \in X_{\epsilon_n}$ for a sequence $\epsilon_n \to \delta$. Then there exist paths
\[
	\xi_n(t) : [0, \epsilon_n] \to M(\infty)
\]
satisfying the requisite equality, and again Arzela-Ascoli guarantees for some subsequence the existence of a pointwise limit
\[
	\xi(t) : [0, \delta] \to M(\infty)
\]
which will again satisfy the requisite equality. Therefore, if we let
\[
	\beta = \sup \set{\delta | X_{\delta} \text{ is nonempty} }
\]
then
\[
	X_{\beta} = \bigcap_{\delta < \beta} X_\delta.
\]
In particular, being a nested intersection of nonempty compact sets, $X_\beta$ is nonempty.

We now show that $\beta < \pi$. To see this, note that $\beta = \pi$ implies in particular that there exist two points $\zeta, \xi$ in $M(\infty)$ such that the angle between $\zeta$ and $\zeta$ when seen from any point is $\pi$. This implies that there exists a vector field $Y$ on $M$ such that for any point $q$, $Y(q)$ points at $\zeta$ and $-Y(q)$ points at $\xi$. The vector field $Y$ is $\cC^1$ by Theorem 1 (ii) in \cite{Esc77}, and the flat strip theorem now shows that the vector field $Y$ is holonomy invariant, so that $M$ is reducible. Thus $\beta < \pi$.

We claim $X_\beta$ is the desired set. We have already shown it is closed, nonempty, and $\Gamma$-invariant, so we have left to show that $X_\beta \neq M(\infty)$. 

Fix a recurrent vector $v \in \Rr_k$; assume for the sake of contradiction that $v \in X_\beta$. Then there exists a path
\[
	\eta(t) : [0, \beta] \to M(\infty)
\]
with $\eta(0) = \eta$ and $\angle_q(\eta(t), \eta(s)) = |t - s|$ for all $t, s \in [0, \beta]$. Let $p = \pi(v)$ be the footpoint of $v$, and let
\[
	\eta_p(t) : [0, \beta] \to S_pP_v
\]
be the lift of $\eta(t)$. Then $\eta_p(t)$ is a geodesic segment in $S_pP_v$. We may choose $0 < \epsilon < A$, where $A$ is as in Lemma \ref{45}, so that $\beta + \epsilon < \pi$. Thus we may extend $\eta_p(t)$ to a geodesic
\[
	\eta_p(t) : [-\epsilon, \beta] \to S_pP_v,
\]
and we may use this to extend $\eta(t)$. By Corollaries \ref{C} and \ref{47}, we have for all $q \in M$
\[
	\angle_q(\eta(t), \eta(s)) = |t - s|,
\]
and so $\eta(-\epsilon) \in X_{\beta + \epsilon}$, contradicting our choice of $\beta$.
\end{proof}

\section{Completion of proof}\label{S_complete}

We now fix a nonempty proper closed $\Gamma$-invariant subset $Z \subseteq M(\infty)$ and define a function $f : SM \to \reals$ by
\[
	f(v) = \min_{\zeta \in Z} \angle_{\pi(v)} (\gamma_v(\infty), \zeta).
\]
It is clear that $f$ is $\Gamma$-invariant, and Lemma \ref{nonincr} gives that $f$ is nondecreasing under the geodesic flow (that is, $f(g^t v) \geq f(v)$). We use the next four lemmas to prove that $f$ is continuous, invariant under the geodesic flow, constant on equivalence classes of asymptotic vectors, and differentiable almost everywhere.

\begin{lem}
$f$ is continuous.
\end{lem}
\begin{proof}
For each $\zeta \in M(\infty)$ define a function $f_{\zeta} : SM \to \reals$ by
\[
	f_{\zeta}(v) = \angle_{\pi(v)} (\gamma_v(\infty), \zeta).
\]
We will show that the family $f_{\zeta}$ is equicontinuous at each $v \in SM$, from which continuity of $f$ follows.

Fix $v \in SM$ and $\epsilon > 0$. There is a neighborhood $U \subseteq SM$ of $v$ and an $a > 0$ such that
\[
	d_a(u, w) = d(\gamma_u(0), \gamma_w(0)) + d(\gamma_u(a), \gamma_w(a))
\]
is a metric on $U$ giving the correct topology. Suppose $w \in U$ with $d_a(v, w) < \epsilon$. For $\zeta \in Z$, let $\zeta_{\pi(v)}, \zeta_{\pi(w)}$ be the vectors at $\pi(v), \pi(w)$, respectively, pointing at $\zeta$. Then
\[
	|d_a(v, \zeta_{\pi(v)}) - d_a(w, \zeta_{\pi(w)})| \leq d_a(v, w) + d_a(\zeta_{\pi(v)}, \zeta_{\pi(w)}) \leq 3 \epsilon,
\]
by the triangle inequality for $d_a$ for the first inequality, and Proposition \ref{OSu76 2} for the second. This gives the desired equicontinuity at $v$.
\end{proof}

\begin{lem}
For $v \in SM$, we have $f(g^t v) = f(v)$ for all $t \in \reals$.
\end{lem}
\begin{proof}
First assume $v$ is recurrent. Fix $t_n \to \infty$ and $\phi_n \in \Gamma$ so that $d\phi_n g^{t_n} v \to v$. Then
\[
	f(d\phi_n g^{t_n} v) = f(g^{t_n} v)
\]
and the sequence $f(g^{t_n}v)$ is therefore an increasing sequence whose limit is $f(v)$ and all of whose terms are bounded below by $f(v)$, so evidently $f(g^{t_n}v) = f(v)$ for all $n$, and it follows that $f(g^t v) = f(v)$ for all $t \in \reals$. 

Now we generalize to arbitrary $v$. Fix $t > 0$ and $\epsilon > 0$. By continuity of $f$ and the geodesic flow, we may choose $\delta > 0$ so that if $u \in SM$ is within $\delta$ of $v$, then 
\[
	|f(u) - f(v)| < \epsilon \text{ and } |f(g^t u) - f(g^t v)| < \epsilon. 
\]
Then choose $u$ recurrent within $\delta$ of $v$ to see that
\[
	|f(g^t v) - f(v)| \leq |f(g^t v) - f(g^t u) | + |f(g^t u) - f(u) | + |f(u) - f(v)| < 2\epsilon.
\]
Since $\epsilon$ was chosen arbitrarily, $f(g^t v) = f(v)$.
\end{proof}

\begin{lem}\label{f asymp}
Let $v, w \in SM$ be arbitrary. If either $v$ and $w$ are asymptotic or $-v$ and $-w$ are asymptotic, then $f(v) = f(w)$.
\end{lem}
\begin{proof}
If $v$ and $w$ are asymptotic, fix by Lemma \ref{seq} $t_n \to \infty$, $w_n \to w$, and $\phi_n \in \Gamma$, such that $(d\phi_n \circ g^{t_n})w_n \to v$. Then since $f$ is continous,
\[
	f(w) = \lim f(w_n) = \lim f((d\phi_n \circ g^{t_n})w_n) = f(v).
\]
On the other hand, if $-v$ and $-w$ are asymptotic, we may fix $t_n \to -\infty$, $w_n \to w$, and $\phi_n \in \Gamma$, such that $(d\phi_n \circ g^{t_n})w_n \to v$, and the exact same argument applies.
\end{proof}

\begin{lem}\label{lip}
$f$ is differentiable almost everywhere. 
\end{lem}
\begin{proof}
Fix $v \in SM$; there is a neighborhood $U$ of $v$ and an $a > 0$ such that 
\[
	d_a(u, w) = d(\gamma_u(0), \gamma_w(0)) + d(\gamma_u(a), \gamma_w(a))
\]
is a metric on $U$ (giving the correct topology). Choose $u, w \in U$, and let $w' \in S_{\pi(u)}M$ be asymptotic to $w$. Then
\[
	|f(u) - f(w)| = |f(u) - f(w')| \leq \angle_{\pi(u)} (u, w') \leq C d_a(u, w'),
\]
for some constant $C$. But note that
\begin{align*}
	d_a(u, w') &= d(\gamma_u(a), \gamma_{w'}(a)) \leq d(\gamma_u(a), \gamma_w(a)) + d(\gamma_w(a), \gamma_{w'}(a)) \\
		&\leq d(\gamma_u(a), \gamma_w(a)) + d(\gamma_w(0), \gamma_{w'}(0)) = d_a(u, w),
\end{align*}
by Proposition \ref{OSu76 2}. Therefore $f$ is Lipschitz with respect to the metric $d_a$ on $U$, and hence differentiable almost everywhere on $U$.
\end{proof}

From here on, the proof follows Ballmann \cite{Bal95}, $\S \text{IV}.6$, essentially exactly. We repeat his steps below for convenience.

We denote by $W^s(v), W^u(v) \subseteq SM$ the weak stable and unstable manifolds through $v$, respectively. Explicitly, $W^s(v)$ is the collection of those vectors asymptotic to $v$, and $W^u(v)$ the collection of those vectors $w$ such that $-w$ is asymptotic to $-v$. 

\begin{lem}
$T_vW^s(v) + T_vW^u(v)$ contains the horizontal subspace of $T_vSM$.
\end{lem}
\begin{proof}
Following Ballmann, given $w \in T_{\pi(v)}M$ we let $B^+(w)$ denote the covariant derivative of the stable Jacobi field $J$ along $\gamma_v$ with $J(0) = w$. That is, $B^+(w) = J'(0)$ where $J$ is the unique Jacobi field with $J(0) = w$ and $J(t)$ bounded as $t \to \infty$. Similarly, $B^-(w)$ is the covariant derivative of the unstable Jacobi field along $\gamma_v$ with $J(0) = w$. In this notation,
\begin{align*}
	T_vW^s(v) &= \set{(w, B^+(w)) | w \in S_{\pi(v)}M} & &\text{and} & T_vW^u(v) &= \set{(w, B^-(w)) | w \in S_{\pi(v)}M}.
\end{align*}
Both $B^+$ and $B^-$ are symmetric (as is shown in Eschenburg-O'Sullivan \cite{EscOsu76}). We let
\[
	E_0 = \set{w \in T_{\pi(v)}M | B^+(w) = B^-(w) = 0}.
\]
Since $B^+$ and $B^-$ are symmetric, they map $T_{\pi(v)}$ into the orthogonal complement $E_0^\perp$ of $E_0$.

The claim of the lemma is that any horizontal vector $(u, 0) \in T_vSM$ can be written in the form
\[
	(u, 0) = (w_1, B^+(w_1)) + (w_2, B^-(w_2)).
\]
This immediately implies $w_2 = u - w_1$, so we are reduced to solving the equation
\[
	-B^-(u) = B^+(w_1) - B^-(w_1),
\]
and for this it suffices to show the operator $B^+ - B^-$ surjects onto $E_0^\perp$, and for this it suffices to show that the restriction
\[
	B^+ - B^- : E_0^\perp \to E_0^\perp
\]
is injective. Assuming $w \in E_0^\perp$, $B^+(w) = B^-(w)$ implies that the Jacobi field $J$ with $J(0) = w$ and $J'(0) = B^+(w) = B^-(w)$ is both stable and unstable, hence bounded, hence, by Proposition \ref{Eb bdd}, parallel; thus $w \in E_0$ and it follows that $w = 0$.
\end{proof}

\begin{cor}
If $c$ is a piecewise smooth horizontal curve in $SM$ then $f \circ c$ is constant.
\end{cor}
\begin{proof}
Obviously it suffices to show the corollary for smooth curves $c$, so we assume $c$ is smooth. By Lemma \ref{lip}, $f$ is differentiable on a set of full measure $D$. By the previous lemma and Lemma \ref{f asymp}, if $\tilde{c}$ is a piecewise smooth horizontal curve such that $\tilde{c}(t) \in D$ for almost all $t$, then $f \circ \tilde{c}$ is constant (since $df(\dot{\tilde{c}}(t)) = 0$ whenever this formula makes sense). 

Our next goal is to approximate $c$ by suitable such curves $\tilde{c}$. Let $l$ be the length of $c$, and parametrize $c$ by arc length. Extend the vector field $\dot{c}(t)$ along $c$ to a smooth horizontal unit vector field $H$ in a neighborhood of $c$. Then there is some smaller neighborhood $U$ of $c$ which is foliated by the integral curves of $H$, and by Fubini (since $D \cap U$ has full measure in $U$), there exists a sequence of smooth horizontal curves $\tilde{c}_r$ such that $\dot{\tilde{c}}_r(t) \in D$ for almost all $t \in [0, l]$, and such that $\tilde{c}_r$ converges in the $\mathscr{C}_0$-topology to $c$. Since $f$ is constant on each curve $\tilde{c}_t$ by the argument in the previous paragraph and $f$ is continuous, we also have that $f$ is constant on $c$.
\end{proof}

Finally, an appeal to the Berger-Simons holonomy theorem proves the result:

\begin{rrthm}
Let $M$ be a complete irreducible Riemannian manifold with no focal points and rank $k \geq 2$. Assume that the $\Gamma$-recurrent geodesics are dense in $M$, where $\Gamma$ is the isometry group of $M$. Then $M$ is a symmetric space of noncompact type.
\end{rrthm}
\begin{proof}
By the previous corollary, the function $f$ is invariant under the holonomy group of $M$. However, it is nonconstant. Thus the holonomy group of $M$ is nontransitive and the Berger-Simons holonomy theorem implies that $M$ is symmetric.
\end{proof}

\section{Fundamental Groups}\label{S_FundGroup}

In this section $M$ is assumed to be a complete simply connected Riemannian manifold without focal points, and $\Gamma$ a discrete, \emph{cocompact} subgroup of isometries of $M$. We will also assume that $\Gamma$ acts properly and freely on $M$, so that $M/\Gamma$ is a closed Riemannian manifold.

Following Prasad-Raghunathan \cite{PraRag72} and Ballmann-Eberlein \cite{BalEbe87}, define for each nonnegative integer $i$ the subset $A_i(\Gamma)$ of $\Gamma$ to be the set of those $\phi \in \Gamma$ such that the centralizer $Z_{\Gamma}(\phi)$ contains a finite index free abelian subgroup of rank no greater than $i$. We sometimes denote $A_i(\Gamma)$ simply by $A_i$ when the group is understood.

We let $r(\Gamma)$ be the minimum $i$ such that $\Gamma$ can be written as a finite union of translates of $A_i$,
\[
	\Gamma = \phi_1 A_i \cup \cdots \cup \phi_k A_i,
\]
for some $\phi_1, \dots, \phi_k \in \Gamma$. Finally, we define the \emph{rank} of $\Gamma$ by
\[
	\rank(\Gamma) = \max\set{r(\Gamma^*) : \Gamma^* \text{ is a finite index subgroup of } \Gamma}.
\]
Ballmann-Eberlein have shown that $\rank(\Gamma) = \rank(M)$ when $M$ has nonpositive curvature. In this section, we generalize their result to no focal points:  

\begin{thm}
Let $M$ be a complete, simply connected Riemannian manifold with no focal points, and let $\Gamma$ be a discrete, cocompact subgroup of isometries of $M$ acting freely and properly. Then $\rank(\Gamma) = \rank(M)$.
\end{thm}

Some remarks are in order. First, the Higher Rank Rigidity Theorem proved earlier in this paper guarantees that $M$ has a de Rham decomposition
\[
	M = M_S \times E_r \times M_1 \times \cdots \times M_l,
\]
where $M_S$ is a higher rank symmetric space, $E_r$ is $r$-dimensional Euclidean space, and $M_i$ is a rank one manifold of no focal points, for $1 \leq i \leq l$. Uniqueness of the de Rham decomposition implies that $\Gamma$ admits a finite index subgroup $\Gamma^*$ which preserves the de Rham splitting. 

We assume for the moment that $M$ has no Euclidean factor, that is, $r = 0$ in the decomposition above. We then have the following lemma:

\begin{lem}
Let $M$ have no flat factors, and let $\Gamma$ be a cocompact subgroup of isometries of $M$. Then $M$ splits as a Riemannian product $M = M_S \times M_1$, where $M_S$ is symmetric and $M_1$ has discrete isometry group.
\end{lem}
\begin{proof}
Let $I_0$ denote the connected component of the isometry group of $M$. By Theorem 3.3 of Druetta \cite{Dru83}, $\Gamma$ has no normal abelian subgroups. Then theorem 3.3 of Farb-Weinberger \cite{FarWei08} shows that $I_0$ is semisimple with finite center, and Proposition 3.1 of the same paper shows that $M$ decomposes as a Riemannian warped product
\[
	N \times_f B,
\]
where $N$ is locally symmetric of nonpositive curvature, and $\Isom(B)$ is discrete. We claim that such a warped product must be trivial, which establishes the lemma.

Thus it suffices to show that a compact nontrivial Riemannian warped product must have focal points: Let $N \times_f B$ be a Riemannian warped product, where $f : B \to \reals_{> 0}$ is the warping function. If $f$ is not constant on $B$, there exists a geodesic $\gamma$ in $B$ such that $f$ is not constant on $\gamma$. Letting $\sigma$ be a unit speed geodesic in $N$, we construct the variation $\Gamma(s, t) = (\sigma(s), \gamma(t))$. It is then easy to see that the variation field $J(t) = \partial_s \Gamma(0, t)$ of this variation satisfies
\[
	||J(t)|| = f(\gamma(t)),
\]
which is bounded but nonconstant, so that $N \times_f B$ must have focal points.
\end{proof}

We also have the following useful splitting theorem:

\begin{prop}
Let $M = M_1 \times M_2$ have no flat factors, and suppose $M_1$ has discrete isometry group. Let $\Gamma$ be a discrete, cocompact subgroup of isometries of $M$. Then $\Gamma$ admits a finite index subgroup that splits as $\Gamma_1 \times \Gamma_2$, where $\Gamma_1 \subseteq \Isom(M_1)$ and $\Gamma_2 \subseteq \Isom(M_2)$.
\end{prop}
\begin{proof}
By the uniqueness of the de Rham decomposition of $M$, we may pass to a finite index subgroup to assume that $\Gamma$ preserves the decomposition $M = M_1 \times M_2$, that is, $\Gamma \subseteq \Isom(M_1) \times \Isom(M_2)$. We let $\pi_i : \Gamma \to \Isom(M_i)$ be the projection maps for $i = 1, 2$. Abusing notation, we also denote by $\pi_i : M \to M_i$ the projection maps.

We wish to show that
\[
	\pi_1(\ker \pi_2) \times \pi_2(\ker \pi_1) 
\] 
is a finite index subgroup of $\Gamma$, and for this it suffices to construct a compact coarse fundamental domain.

Let $F$ be a compact fundamental domain for the action of $\Gamma$, and let $H_1 \subseteq M_1$ be the Dirichlet fundamental domain for $\pi_1 \Gamma$. The set of all $a \in \pi_1 \Gamma$ such that $a H_1 \cap \pi_1 F \neq \emptyset$ is finite; denote its elements by $a_1, \dots, a_k$, and fix $b_1, \dots, b_k \in \Isom(M_2)$ such that $(a_i, b_i) \in \Gamma$ for each $i$. Consider the compact set
\[
	K_2 = (a_1^{-1}, b_1^{-1}) F \cup \cdots \cup (a_k^{-1}, b_k^{-1}) F;
\]
we claim $H_1 \times M_2 \subseteq (\ker \pi_1)K_2$.

To see this let $q_1 \in H_1, q_2 \in M_2$. There exists $(p_1, p_2) \in F$ and some $\gamma \in \Gamma$ such that $\gamma(p_1, p_2) = (q_1, q_2)$, and $\gamma$ has the form $\gamma = (a_i^{-1}, \gamma_2)$ for some $\gamma_2 \in \Isom(M_2)$. But then
\[
	(q_1, q_2) \in (1, \gamma_2 b_i)(a_i^{-1}, b_i^{-1}) F \subseteq (\ker\pi_1)K_2.
\]
This establishes our claim.

Note that discreteness of $\Isom(M_1)$ was used only to show $\pi_1 \Gamma$ is discrete. Thus we may now repeat the argument with $\pi_2 \ker \pi_1$ (which we have just established is cocompact) in place of $\pi_1 \Gamma$, and $\pi_1 \ker \pi_2$ in place of $\pi_2 \ker \pi_1$: We let $H_2$ be a fundamental domain for $\pi_2 \ker \pi_1$, and we obtain a compact set $K_1$ such that $M_1 \times H_2 \subseteq (\ker \pi_2)K_1$.

It is now not difficult to see that $K_1$ is a coarse fundamental domain for $\pi_1(\ker \pi_2) \times \pi_2(\ker \pi_1)$.
\end{proof}

As a consequence of these lemmas and the Rank Rigidity Theorem proven earlier in this paper, if $M$ has no flat factors, then it admits a decomposition
\[
	M = M_S \times M_1 \times \cdots \times M_l,
\]
where each of the $M_i, 1 \leq i \leq l$, has rank one and discrete isometry group, and furthermore that our group $\Gamma$ has a finite index subgroup $\Gamma^*$ splitting as
\[
	\Gamma^* = \Gamma_S \times \Gamma_1 \times \cdots \times \Gamma_l.
\]
Ballmann-Eberlein have shown (Theorem 2.1 in \cite{BalEbe87}) that $\rank(\Gamma^*) = \rank(\Gamma)$ and that the rank of $\Gamma^*$ is the sum of the rank of $\Gamma_S$ and the ranks of the $\Gamma_i$. Prasad-Raghunathan have shown that $\rank(\Gamma_S) = \rank(M_S)$ in \cite{PraRag72}. 

Therefore, to finish the proof in the case where $M$ has no flat factors, we need to show that $\rank(\Gamma) = \rank(M)$ in the case where $M$ is irreducible, rank one, and has discrete isometry group. We proceed to do this now; we mimic the geometric construction of Ballmann-Eberlein, and for this we first generalize a number of lemmas about rank one geodesics in manifolds of nonpositive curvature to the no focal points case.

\subsection{Rank one $\Gamma$-periodic vectors.}

The following series of lemmas generalizes the work of Ballmann in \cite{Bal82}. As in that paper, we will be interested in geodesics $\gamma$ that are $\Gamma$-periodic, i.e., such that there exists a $\phi \in \Gamma$ and some $a \in \reals$ with $\phi \circ \gamma(t) = \gamma(t + a)$ for all $t$. Such a geodesic $\gamma$ will be called \emph{axial}, and $\phi$ will be called an \emph{axis} of $\gamma$ with \emph{period} $a$. 

We denote by $\overline{M}$ the union $M \cup M(\infty)$, and for each tangent vector $v$ and each $\epsilon$ we define the cone $C(v, \epsilon) \subseteq \overline{M}$ to be the set of those $x \in \overline{M}$ such that the geodesic from $\pi(v)$ to $x$ makes angle less than $\epsilon$ with $v$. The sets $C(v, \epsilon)$ together with the open subsets of $M$ form a subbasis for a topology on $\overline{M}$, called the \emph{cone topology}. Goto \cite{Got79} has shown that the cone topology is the unique topology on $\overline{M}$ with the property that for any $p \in M$, the exponential map is a homeomorphism of $\overline{T_pM}$ with $\overline{M}$ (where the former is given the cone topology).

If $p, q \in M$, we denote by $\gamma_{pq}$ the unit speed geodesic through $p$ and $q$ with $\gamma(0) = p$. We denote by $\gamma(\infty)$ the element of $M(\infty)$ at which $\gamma$ points, and analogously for $\gamma(-\infty)$. Note that if $\gamma$ is a geodesic and $t_n \to \infty$, then $\gamma(t_n) \to \gamma(\infty)$ in the cone topology on $\overline{M}$. Moreover, if $p_n \in \overline{M}$ and $p_n \to \zeta \in M(\infty)$, then for $p \in M$ the geodesics $\gamma_{pp_n}$ converge to $\gamma_{p\zeta}$. This follows from considering $\overline{T_pM}$ and the result of Goto cited above. More generally, we have the following lemma:

\begin{lem}\label{A_Aa}
Let $p, p_n \in M$ with $p_n \to p$, and let $x_n, \zeta \in \overline{M}$ with $x_n \to \zeta$. Then $\dot{\gamma}_{p_nx_n}(0) \to \dot{\gamma}_{p\zeta}(0)$.
\end{lem}
\begin{proof}
First pass to any convergent subsequence of $\dot{\gamma}_{p_nx_n}(0)$; say this subsequence converges to $\dot{\gamma}_{p\xi}(0)$, where $\xi \in M(\infty)$. Suppose for the sake of contradiction that $\xi \neq \zeta$. Let $c = d(\gamma_{p\zeta}(1), \gamma_{p\xi}(1)) > 0$. By the remarks preceding the lemma, we may choose $n$ large enough so that each of
\[
	d(p_n, p), \; d(\gamma_{p_nx_n}(1), \gamma_{p\xi}(1)), \text{ and } d(\gamma_{px_n}(1), \gamma_{p\zeta}(1))
\]
is strictly smaller than $c/3$. Proposition \ref{OSu76 1} shows that $d(\gamma_{p_nx_n}(1), \gamma_{px_n}(1)) < c/3$, and the triangle inequality gives the desired contradiction:
\begin{align*}
	c &= d(\gamma_{p\zeta}(1), \gamma_{p\xi}(1))  \\
		&\leq d(\gamma_{p\zeta}(1), \gamma_{px_n}(1)) + d(\gamma_{px_n}(1), \gamma_{p_nx_n}(1)) + d(\gamma_{p_nx_n}(1), \gamma_{p\xi}(1)) 
		\\ &< c.
\end{align*}
\end{proof}

Finally, we say that a geodesic $\gamma$ \emph{bounds a flat half-strip of with $c$} if there exists an isometric immersion $\Phi: \reals \times [0, c) \to M$ such that $\Phi(t, 0) = \gamma(t)$, and that $\gamma$ \emph{bounds a flat half-plane} if there exists such $\Phi$ with $c = \infty$.

\begin{lem}\label{A_2.1.a}
Let $\gamma$ be a geodesic, and suppose there exist
\begin{align*}
	p_k &\in C(-\dot{\gamma}(0), 1/k) \cap M & q_k &\in C(\dot{\gamma}(0), 1/k) \cap M
\end{align*}
such that $d(\gamma(0), \gamma_{p_k q_k}) \geq c > 0$ for all $k$. Then $\gamma$ is the boundary of a flat half-strip of width $c$. 
\end{lem}
\begin{proof}
For each $k$ let $\tilde{p}_k, \tilde{q}_k$ be the points on $\gamma$ closest to $p_k, q_k$, respectively. Let $b_k(s)$ be a smooth path with $b_k(0) = \tilde{p}_k, \; b_k(1) = p_k$, and similarly let $c_k(s)$ be a smooth path with $c_k(0) = \tilde{q}_k, \; c_k(1) = q_k$. We may further choose $b_k$ so that the angle
\[
	\angle_{\gamma(0)}(\tilde{p}_k, b_k(s))
\]
is an increasing function of $s$, and similarly for $c_k$. Finally, let $\sigma_{k,s}(t)$ be the unit speed geodesic through $b_k(s)$ and $c_k(s)$, parameterized so that $\sigma_{k,0}(0) = \gamma(0)$, and such that $s \mapsto \sigma_{k,s}(0)$ is a continuous path in $M$.

By hypothesis, $d(\sigma_{k,1}(0), \gamma(0)) \geq c$. Thus there exists $s_k, 0 < s_k \leq 1$, with $d(\sigma_{k, s_k}(0), \gamma(0)) = c$. Passing to a subsequence, we may assume the geodesics $\sigma_{k, s_k}$ converge as $k \to \infty$ to a geodesic $\sigma$ with $d(\sigma(0), \gamma(0)) = c$.

Finally, any convergent subsequence of $b_k(s_k)$, or of $c_k(s_k)$, must converge to a point on $\gamma$, or one of the endpoints of $\gamma$. However, Lemma \ref{A_Aa}, and the fact that $\sigma \neq \gamma$, shows that the only possibility is $b_k(s_k) \to \gamma(-\infty)$ and $c_k(s_k) \to \gamma(\infty)$. Another application of Lemma \ref{A_Aa} shows that $\sigma$ is parallel to $\gamma$. The flat strip theorem now gives the result.
\end{proof}

\begin{lem}\label{A_2.1.b}
Let $\gamma$ be rank one, and $c > 0$. Then there exists $\epsilon > 0$ such that if $x \in C(-\dot{\gamma}(0), \epsilon)$, $y \in C(\dot{\gamma}(0), \epsilon)$, then there is a geodesic connecting $x$ and $y$.

Furthermore, if $\sigma$ is a geodesic with $\sigma(-\infty) \in C(-\dot{\gamma}(0), \epsilon)$ and $\sigma(\infty) \in C(\dot{\gamma}(0), \epsilon)$, then $\sigma$ does not bound a flat half plane, and $d(\gamma(0), \sigma) \leq c$.
\end{lem}
\begin{proof}
By Lemma \ref{A_2.1.a}, there exists $\epsilon > 0$ such that $d(\gamma_{pq}, \gamma(0)) \leq c$ if $p \in C(-\dot{\gamma}(0), \epsilon) \cap M$ and $q \in C(\dot{\gamma}(0), \epsilon) \cap M$. We choose sequences $p_n \to x$ and $q_n \to y$; then some subsequence of $\gamma_{p_n q_n}$ converges to a geodesic connecting $x$ and $y$.

To prove the second part, note that all geodesics $\tau$ with endpoints in $C(-\dot{\gamma}(0), \epsilon)$ and $C(\dot{\gamma}(0), \epsilon)$ satisfy $d(\gamma(0), \tau) \leq c$ by choice of $\epsilon$. However, if $\sigma$ bounds a flat half-plane then there are geodesics $\tau_n$ with the same endpoints as $\sigma$ but with $\tau_n \to \infty$, a contradiction.
\end{proof}

As a corollary of the above, we see that if $\gamma$ is rank one and $\gamma_n$ is a sequence of geodesics with $\gamma_n(-\infty) \to \gamma(-\infty)$ and $\gamma_n(\infty) \to \gamma(\infty)$, then $\gamma_n \to \gamma$.

\begin{lem}\label{A_B}
Let $\gamma$ be a recurrent geodesic, and suppose $\phi_n$ is a sequence of isometries such that $d\phi_n(\dot{\gamma}(t_n)) \to \dot{\gamma}(0)$, where $t_n$ increases to $\infty$. Further suppose that there exists $x, \zeta \in M(\infty)$ with $\phi_n(x) \to \zeta$, where $\zeta \neq \gamma(\infty)$ and $\zeta \neq \gamma(-\infty)$. Then $\gamma$ is the boundary of a flat half plane $F$, and $\zeta \in F(\infty)$. 
\end{lem}
\begin{proof}
For each $s \in \reals$ let $\tau_s$ be the geodesic with $\tau_s(0) = \gamma(s)$ and $\tau_s(\infty) = x$, and let $\sigma_s$ be the geodesic with $\sigma_s(0) = \gamma(s)$ and $\sigma_s(\infty) = \zeta$. Fix $t > 0$.

We first claim that for each $\epsilon > 0$, there exists an infinite subset $L(\epsilon) \subseteq \naturals$ such that for each $N \in L(\epsilon)$ there exists an infinite subset $L_N(\epsilon) \subseteq \naturals$ such that for $n \in L_N(\epsilon)$,
\[
	d(\tau_{t_N}(t), \tau_{t_n}(t)) \geq t_n - t_N - \epsilon.
\]
Let us first show this claim.

By passing to a subsequence, we may assume
\[
	d(\phi_n\gamma(t_n), \gamma(0)) < \epsilon/3 \text{ and } d(\phi_n \tau_{t_n}(t), \sigma_0(t)) < \epsilon/3
\]	
for all $n \geq 1$; the second inequality follows from recurrence of $\gamma$, the fact that $\phi_n(x) \to \zeta$, and Proposition \ref{OSu76 2}.

Assume for the sake of contradiction that our claim is false; then again by passing to a subsequence, we may assume that for $m > n \geq 1$
\[
	d(\tau_{t_n}(t), \tau_{t_m}(t)) < t_m - t_n - \epsilon.
\]
From this and the previous inequality, we conclude that for $m > n \geq 1$
\[
	d(\phi_n^{-1}\sigma_0(t), \phi_m^{-1}\sigma_0(t)) < t_m - t_n - \epsilon/3.
\]
Choose $l$ such that $l\epsilon/3 > 2t + \epsilon$. Then
\begin{align*}
	d(\gamma(t_1), \gamma(t_l)) &\leq d(\gamma(t_1), \phi_1^{-1}\gamma(0)) + d(\phi_1^{-1}\gamma(0), \phi_1^{-1}\sigma_0(t)) + \sum_{i = 1}^{l-1} d(\phi_i^{-1}\sigma_0(t), \phi_{i+1}^{-1}\sigma_0(t)) \\
		&+ d(\phi_l^{-1}\sigma_0(t), \phi_l^{-1}\gamma(0)) + d(\phi_l^{-1}\gamma(0), \gamma(t_l)) \\
		&< \epsilon/3 + t + \sum_{i = 1}^{l-1} (t_{i+1} - t_i - \epsilon/3) + t + \epsilon/3 \\
		&\leq 2t + \epsilon - l\epsilon/3 + t_l - t_1 \\
		&< t_l - t_1,
\end{align*}
contradicting the fact that $\gamma$ is length minimizing. This proves our claim.

The next step of the proof is to show that for $s > 0$
\[
	d(\sigma_0(t), \sigma_s(t)) = s.
\]
Fix such $s$. Note that $d(\sigma_0(t), \sigma_s(t)) \leq s$ by Proposition \ref{OSu76 2}. Suppose for the sake of contradiction that
\[
	d(\sigma_0(t), \sigma_s(t)) = s - 3\epsilon
\]
for some $\epsilon > 0$. Choose $N \in L(\epsilon)$ large enough such that
\[
	d(\phi_N \tau_{t_N}(t), \sigma_0(t)) < \epsilon \text{ and } d(\phi_N \tau_{t_N + s}(t), \sigma_s(t)) < \epsilon.
\]
As before, that this can be done follows from recurrence of $\gamma$, the fact that $\phi_n(x) \to \zeta$, and Proposition \ref{OSu76 2}. Then if $n \in L_N(\epsilon)$ with $t_n > t_N + s$, we find
\begin{align*}
	d(\tau_{t_N}(t), \tau_{t_n}(t)) &= d(\phi_N \tau_{t_N}(t), \phi_N \tau_{t_n}(t)) \\
	&\leq d(\phi_N \tau_{t_N}(t), \sigma_0(t)) + d(\sigma_0(t), \sigma_s(t)) \\
	&+ d(\sigma_s(t), \phi_N\tau_{t_N + s}(t)) + d(\phi_N \tau_{t_N + s}(t), \phi_N \tau_{t_n}(t)) \\
	&< \epsilon + (s - 3\epsilon) + \epsilon + t_n - (t_N + s) \\
	&= t_n - t_N - \epsilon,
\end{align*}
contradicting the definitions of $L(\epsilon), L_N(\epsilon)$. Hence
\[
	d(\sigma_s(t), \sigma_0(t)) = s
\]
as claimed. In fact, the above argument shows that for all $r, s \in \reals$
\[
	d(\sigma_r(t), \sigma_s(t)) = |r - s|.
\]

We now complete the proof. Lemma 2 in O'Sullivan \cite{OSu76} shows that the curves $\theta_t$ defined by $\theta_t(s) = \sigma_s(t)$ are geodesics, and they are evidently parallel to $\gamma$. Thus the flat strip theorem guarantees for each $t$ the existence of a flat $F_t$ containing $\gamma$ and $\theta_t$; since $F_t$ is totally geodesic, it contains each of the geodesics $\sigma_s$. (We remark, of course, that all the $F_t$ coincide.)
\end{proof}

\begin{cor}\label{A_E}
Let $\gamma$ be a recurrent geodesic, and suppose there exists $x \in M(\infty)$ such that $\angle_{\gamma(t)}(x, \gamma(\infty)) = \epsilon$ for all $t$, where $0 < \epsilon < \pi$. Then $\gamma$ is the boundary of a flat half-plane.
\end{cor}
\begin{proof}
If $\phi_n$ is a sequence of isometries such that $\phi_n \dot{\gamma}(t_n) \to \dot{\gamma}(0)$, for $t_n \to \infty$, one sees that any accumulation point $\zeta$ of $\phi_n(x)$ in $M(\infty)$ must satisfy $\angle_{\gamma(0)}(\gamma(\infty), \zeta) = \epsilon$, and so the previous lemma applies.
\end{proof}

\begin{cor}\label{A_2.4}
Let $\phi$ be an isometry with axis $\gamma$ and period $a$. Suppose $B \subseteq M(\infty)$ is nonempty, compact, $\phi(B) \subseteq B$, and neither $\gamma(\infty)$ nor $\gamma(-\infty)$ is in $B$. Then $\gamma$ bounds a flat half plane.
\end{cor}
\begin{proof}
Take $\phi_n = \phi^n$ and $t_n = na$, along with the recurrent geodesic $-\gamma$, in Lemma \ref{A_B}.
\end{proof}

\begin{lem}\label{A_2.5}
Let $\phi$ be an isometry with rank one axis $\gamma$ and period $a$. Then for all $\epsilon, \delta$ with $0 < (\epsilon, \delta) < \pi$ and all $t \in \reals$ there exists $s$ with
\[
	\overline{C(\dot{\gamma}(s), \delta)} \subseteq C(\dot{\gamma}(t), \epsilon).
\]
\end{lem}
\begin{proof}
Suppose otherwise; then there exists such $\epsilon, \delta, t$ such that for all $s$ the above inclusion does not hold. In particular we may choose for each $n$ a point $z_n$ with
\begin{align*}
	z_n &\in \overline{C(\dot{\gamma}(na), \delta)} & z_n &\notin C(\dot{\gamma}(t), \epsilon).
\end{align*}
Then if we set $x_n = \phi^{-n}(z_n)$, we have $x_n \in \overline{C(\dot{\gamma}(0), \delta)}$, and none of $x_n, \phi(x_n), \dots, \phi^n(x_n)$ is in $C(\dot{\gamma}(t), \epsilon)$. 

Thus if we let $B$ be the set
\[
	B = \set{x \in M(\infty) \cap \overline{C(\dot{\gamma}(0), \delta)} : \phi^n(x) \notin C(\dot{\gamma}(t), \epsilon) \text{ for all } n},
\]
we see that $B$ is nonempty (it contains any accumulation point of $x_n$) and satisfies the other requirements of Corollary \ref{A_2.4}, so $\gamma$ is the boundary of a flat half plane. 
\end{proof}

\begin{thm}\label{A_2.2}
Let $\phi$ be an isometry with axis $\gamma$ and period $a$. The following are equivalent:
\begin{enumerate}
\item $\gamma$ is not the boundary of a flat half plane;
\item Given $\overline{M}$-neighborhoods $U$ of $\gamma(-\infty)$ and $V$ of $\gamma(\infty)$, there exists $N \in \naturals$ with $\phi^n(\overline{M} - U) \subseteq V$ and $\phi^{-n}(\overline{M} - V) \subseteq U$ whenever $n \geq N$; and
\item For any $x \in M(\infty)$ with $x \neq \gamma(\infty)$, there exists a geodesic joining $x$ and $\gamma(\infty)$, and none of these geodesics are the boundary of a flat half plane.
\end{enumerate}
\end{thm}
\begin{proof}
$(1 \Rightarrow 2)$ By Lemma \ref{A_2.5} we can find $s \in \reals$ with
\begin{align*}
	&\overline{C(-\dot{\gamma}(-s), \pi/2)} \subseteq U, & &\overline{C(\dot{\gamma}(s), \pi/2)} \subseteq V.
\end{align*}
If $Na > 2s$ then for $n \geq N$
\begin{align*}
	\phi^n(\overline{M} - U) &\subseteq \phi^n(\overline{M} - C(-\dot{\gamma}(-s), \pi/2)) \\
				&\subseteq \overline{C(\dot{\gamma}(s), \pi/2)} \\
				&\subseteq V,
\end{align*}
and analogously for $U$ and $V$ swapped.

$(1 \Rightarrow 3)$ By Lemma \ref{A_2.1.b} we can find $\epsilon > 0$ such that for $y \in C(-\dot{\gamma}(0), \epsilon)$ there exists a geodesic from $y$ to $\gamma(\infty)$ which does not bound a flat half plane. But by $(2)$ we can find $n$ such that $\phi^{-n}(x) \in C(-\dot{\gamma}(0), \epsilon)$.

$(2 \Rightarrow 1)$ and $(3 \Rightarrow 1)$ are obvious (by checking the contrapositive).
\end{proof}

\begin{prop}\label{A_D}\label{A_2.13}
If $\gamma$ is rank one and $U, V$ are neighborhoods of $\gamma(-\infty)$ and $\gamma(\infty)$, then there exists an isometry $\phi \in \Gamma$ with rank one axis $\sigma$, where $\sigma(-\infty) \in U$ and $\sigma(\infty) \in V$.
\end{prop}
\begin{proof}
Since $\Gamma$-recurrent vectors are dense in $SM$, we may assume $\gamma$ is recurrent, and take $\phi_n \in \Gamma$, $t_n \to \infty$, such that $d\phi_n \dot{\gamma}(t_n) \to \dot{\gamma}(0)$. By Lemma \ref{A_2.1.b} we may replace $U$ and $V$ by smaller neighborhoods such that for any $x \in U, y \in V$, there exists a rank one geodesic joining $x$ and $y$. By the flat strip theorem, such a geodesic is unique.

The argument in the proof of Lemma 2.13 in Ballmann \cite{Bal82}, using Corollary \ref{A_E} in place of Ballmann's Proposition 1.2, shows that for sufficiently large $n$, $\phi_n$ has fixed points $x_n \in U$ and $y_n \in V$. Then $\phi_n$ must fix the oriented geodesic $\sigma_n$ from $x_n$ to $y_n$. Since $d(\phi_n \gamma(0), \gamma(0)) \to \infty$, but $d(\sigma_n, \gamma(0))$ is uniformly bounded (again by Lemma \ref{A_2.1.b}), $\phi_n$ must act as a nonzero translation on $\sigma_n$ for large enough $n$. 
\end{proof}

\begin{cor}\label{A_D_c}
Rank one $\Gamma$-periodic vectors are dense in the set of rank one vectors.
\end{cor}

\subsection{The geometric construction.}

Our goal in this subsection is to prove the following:

\begin{thm}\label{A_3.1}
Let $M$ have rank one, and let $\Gamma$ be a discrete subgroup of isometries of $M$ such that $\Gamma$-recurrent vectors are dense in $M$. Then $r(\Gamma) = 1$.
\end{thm}

Our method is simply to show that the Ballmann-Eberlein construction works equally well in the setting of no focal points. Thus we define
\[
	B_1(\Gamma) = \set{\phi \in \Gamma : \phi \text{ translates a rank one geodesic }}.
\]

\begin{lem}\label{A_C}
$B_1(\Gamma) \subseteq A_1(\Gamma)$.
\end{lem}
\begin{proof}
For $\phi \in B_1(\Gamma)$ translating $\gamma$, the flat strip theorem guarantees that $\gamma$ is the unique rank one geodesic translated by $\phi$. Thus every element of $Z_{\Gamma}(\phi)$ leaves $\gamma$ invariant. Since $\Gamma$ is discrete, $Z_{\Gamma}(\phi)$ must therefore contain an infinite cyclic group of finite index.
\end{proof}

As in Ballmann-Eberlein a point $x \in M(\infty)$ is called \emph{hyperbolic} if for any $y \neq x$ in $M(\infty)$, there exists a rank one geodesic joining $y$ to $x$. By Theorem \ref{A_2.2}, any rank one axial geodesic has hyperbolic endpoints; thus Corollary \ref{A_D_c} implies that the set of hyperbolic points is dense in the open set of $M(\infty)$ consisting of endpoints of rank one vectors.

\begin{lem}\label{A_3.5}
Let $p \in M$, let $x \in M(\infty)$ be hyperbolic, and let $U^*$ be a neighborhood of $x$ in $\overline{M}$. Then there exists a neighborhood $U$ of $x$ in $\overline{M}$ and $R > 0$ such that if $\sigma$ is a geodesic with endpoints in $U$ and $\overline{M} - U^*$, then $d(p, \sigma) \leq R$.
\end{lem}
\begin{proof}
Repeat the argument of Ballmann-Eberlein, Lemma 3.5 \cite{BalEbe87}. (This proof references Lemma 3.4 of the same paper, which follows immediately from our Lemma \ref{A_2.1.b}.)
\end{proof}

\begin{lem}\label{A_3.6}
Let $x \in M(\infty)$ be hyperbolic, and $U^*$ a neighborhood of $x$ in $M(\infty)$. Then there exists a neighborhood $U \subseteq M(\infty)$ of $x$ such that for all $x^* \in U, y^* \in M(\infty) - U^*$, there exists a rank one geodesic between $x^*$ and $y^*$.
\end{lem}
\begin{proof}
Repeat the argument of Ballmann-Eberlein, Lemma 3.6 \cite{BalEbe87}.
\end{proof}

\begin{lem}\label{A_3.8}
Let $x, y$ be distinct points in $M(\infty)$ with $x$ hyperbolic, and suppose $U_x$ and $U_y$ are neighborhoods of $x$ and $y$, respectively. Then there exists an isometry $\phi \in \Gamma$ with
\begin{align*}
	&\phi(\overline{M} - U_x) \subseteq U_y & &\phi^{-1}(\overline{M} - U_y) \subseteq U_x.
\end{align*}
\end{lem}
\begin{proof}
By Proposition \ref{A_D} there is a $\Gamma$-periodic geodesic with endpoints in $U_x$ and $U_y$; then apply Theorem \ref{A_2.2}.
\end{proof}

\begin{lem}\label{A_3.9}
Let $x \in M(\infty)$ be hyperbolic, $U^* \subseteq \overline{M}$ a neighborhood of $x$, and $p \in M$. Then there exists a neighborhood $U \subseteq \overline{M}$ of $x$ such that if $\phi_n$ is a sequence of isometries with $\phi_n(p) \to z \in M(\infty) - U^*$, then
\[
	\sup_{u \in U} \angle_{\phi_n(p)}(p, u) \to 0 \text{ as } n \to \infty.
\]
\end{lem}
\begin{proof}
By Lemma \ref{A_3.5} there exists $R > 0$ and a neighborhood $U \subseteq \overline{M}$ of $x$ such that if $\sigma$ is a geodesic with endpoints in $U$ and $\overline{M} - U^*$ then $d(p, \sigma) \leq R$.

Let $x_n \in U$ be an arbitrary sequence, and for each $n$ let $\sigma_n$ be the geodesic through $x_n$ with $\sigma_n(0) = \phi_n(p)$. Denote by $b_n$ be the point on $\sigma_n$ closest to $p$, and let $\gamma_n$ be the geodesic through $p$ with $\gamma_n(0) = \phi_n(p)$.

By construction $d(p, b_n) \leq R$, and so we also have $d(\phi_n^{-1}(p), \phi_n^{-1}(b_n)) \leq R$. It follows that any subsequential limit of $\phi_n^{-1} \sigma_n$ is asymptotic to any subsequential limit of $\phi_n^{-1} \gamma_n$. In particular
\[
	\angle_{\phi_n(p)}(p, x_n) = \angle_p(\phi_n^{-1}(p), \phi_n^{-1}(x_n)) \to 0,
\]
from which the lemma follows.
\end{proof}

\begin{thm}
If $M$ is a rank one manifold without focal points and $\Gamma$ is a discrete subgroup of isometries of $M$, then $r(\Gamma) = 1$.
\end{thm}
\begin{proof}
This now follows, with at most trivial modifications, from the argument in the proof of Theorem 3.1 in Ballmann-Eberlein \cite{BalEbe87}. (This argument uses Lemma 3.10 of that paper; their proof of that lemma works in no focal points as well, when combined with the lemmas we have proven above.)
\end{proof}

\subsection{Completion of the Proof.}

We now write $M = E_r \times M_1$, where $E_r$ is a Euclidean space of dimension $r$ and $M_1$ is a manifold of no focal points with no flat factors. We fix a discrete, cocompact subgroup $\Gamma$ of isometries of $M$; by uniqueness of the de Rham decomposition, $\Gamma$ respects the factors of the decomposition $M = E_r \times M_1$. In light of this, we freely write elements of $\Gamma$ as $(\gamma_e, \gamma_1)$, where $\gamma_e$ is an isometry of $E_r$ and $\gamma_1$ an isometry of $M_1$.

In this section we work with Clifford transformations, which are isometries $\phi$ of $M$ such that $d(p, \phi(p))$ is constant for $p \in M$. For a group $\Gamma'$ of isometries of $M$, we denote by $C(\Gamma')$ the set of Clifford transformations in $\Gamma'$, and by $Z(\Gamma')$ the center of $\Gamma'$. Theorem 2.1 of Druetta \cite{Dru83} shows that a Clifford transformation of $M = E_r \times M_1$ has the form $(\phi, \id)$, where $\phi$ is a translation of $E_r$. 

We begin by proving the following generalization (to no focal points) of a lemma in Eberlein \cite{Ebe82}:

\begin{lem}
$\Gamma$ admits a finite index subgroup $\Gamma_0$ such that for any finite index subgroup $\Gamma^*$ of $\Gamma_0$, we have $Z(\Gamma^*) = C(\Gamma^*)$.
\end{lem}
\begin{proof}
Our proof is essentially the same as Eberlein's. We let $\Gamma_0$ be the centralizer $Z_{\Gamma}(C(\Gamma))$; note that $\Gamma_0$ is just the subgroup of $(\gamma_e, \gamma_1) \in \Gamma$ such that $\gamma_e$ is a Euclidean translation. Since $C(\Gamma)$ is just the set of those elements of the form $(\gamma_e, \id)$ where $\gamma_e$ is a translation, $C(\Gamma) \subseteq \Gamma_0$, and we have
\[
	C(\Gamma_0) = C(\Gamma) \subseteq Z(\Gamma_0).
\]
We claim $\Gamma_0$ is finite index in $\Gamma$. This follows from a trivial modification of the argument in Lemma 3 in Yau \cite{Yau71}, noting that $\Gamma_0$ is normal in $\Gamma$ and the projection of $\Gamma_0$ to its first factor is a lattice of translations of $E_r$.

Now let $\Gamma^*$ be a finite index subgroup of $\Gamma_0$. Theorem 3.2 of Druetta \cite{Dru83} gives $Z(\Gamma^*) \subseteq C(\Gamma^*)$. On the other hand,
\[
	C(\Gamma^*) \subseteq C(\Gamma_0) \cap \Gamma^* \subseteq Z(\Gamma_0) \cap \Gamma^* \subseteq Z(\Gamma^*),
\]
so $C(\Gamma^*) = Z(\Gamma^*)$.
\end{proof}

We also need the following, which is Lemma A of Eberlein's paper \cite{Ebe83}:
\begin{lem}
Let $M = E_r \times M_1$ as above, and let $p_1: \Isom(M) \to \Isom(M_1)$ be projection onto the second factor. If $\Gamma$ is a discrete subgroup of $M$ such that $\Gamma$-recurrent vectors are dense in $SM$, then $p_1(\Gamma)$ is discrete in $\Isom(M_1)$.
\end{lem}
\begin{proof}
We sketch the argument of Eberlein, indicating the necessary changes for no focal points. For details see the proof of Lemma A in \cite{Ebe83}. 

Let $A$ denote the subgroup of translations of $E_r$ and let $G$ be the closure in $\Isom(M)$ of $\Gamma A$. $A$ is a closed normal abelian subgroup of $G$ and the connected component $G_0$ of the identity of $G$ is solvable. Then $p_1(G_0)$ is solvable, and we let $A^*$ be the last nonidentity subgroup in its derived series; then $A^*$ is an abelian normal subgroup of $p_1(G_0)$, and by Theorem 3.3 of Druetta \cite{Dru83}, $A^*$ must be trivial. Hence $p_1(G_0) = \set{\id}$. 

It now follows as in \cite{Ebe83} that $p_1(\Gamma)$ is discrete.
\end{proof}

\begin{thm}
Let $M$ be a complete, simply connected Riemannian manifold with no focal points, and let $\Gamma$ be a discrete, cocompact subgroup of isometries of $M$ acting freely and properly. Then $\rank(\Gamma) = \rank(M)$.
\end{thm}
\begin{proof}
With the above results in hand, the rest of the argument is now identical to the proof of Theorem 3.11 in Ballmann-Eberlein. (We remark that this proof cites the main theorem of Eberlein's paper \cite{Ebe83}; this theorem follows from Druetta, Theorem 3.3 \cite{Dru83}.)
\end{proof}

\bibliographystyle{alpha}
\bibliography{geometry}
\end{document}